\documentclass[12pt]{article}

\usepackage{amsmath, amsthm, amssymb, stmaryrd}
\usepackage[all]{xy}
\usepackage{fullpage}
\usepackage{titlesec}
\usepackage{mathrsfs}
\usepackage{amsbsy}
\usepackage{turnstile}
\usepackage{bbm}
\usepackage{yhmath}
\usepackage{tensor}
\usepackage{graphics}
\usepackage{enumitem}
\usepackage{calligra}
\usepackage{verbatim}
\usepackage{setspace}
\usepackage{hhline}
\usepackage{array}

\xyoption{rotate}

\usepackage[pdftex,bookmarks=true]{hyperref}
\usepackage[svgnames]{xcolor}
\hypersetup{
    linkbordercolor = {PaleTurquoise},
    citebordercolor = {PaleGreen},
}

\usepackage{titletoc}

\titlecontents{section}
[0pt]                                               
{}%
{\contentsmargin{0pt}                               
    \thecontentslabel\enspace%
    }
{\contentsmargin{0pt}}                        
{\titlerule*[.5pc]{.}\contentspage}                 
[]                                                  

\makeatletter
\newcommand*{\@old@slash}{}\let\@old@slash\slash
\def\slash{\relax\ifmmode\delimiter"502F30E\mathopen{}\else\@old@slash\fi}
\makeatother

\titleformat{\section}{\normalsize\bfseries}{\thesection}{1em}{}
\titleformat{\subsection}{\normalsize\bfseries}{\thesubsection}{1em}{}

\numberwithin{equation}{subsection}

\theoremstyle{plain}

\newtheorem{prop}[subsection]{Proposition}

\newtheorem{cor}[subsection]{Corollary}
\newtheorem{thm}[subsection]{Theorem}

\theoremstyle{definition}

\newtheorem{defn}[subsection]{Definition}
\newtheorem{exa}[subsection]{Example}
\newtheorem{rem}[subsection]{Remark}
\newtheorem{para}[subsection]{}
\newtheorem{assn}[subsection]{Assumption}
\newtheorem{notn}[subsection]{Notation}

\newdir{ >}{{}*!/-10pt/@{>}}

\DeclareMathAlphabet{\mathpzc}{OT1}{pzc}{m}{it}
\DeclareMathAlphabet{\mathcalligra}{T1}{calligra}{m}{n}

\newcommand{\pref}[1]{\textnormal{(\ref{#1})}}

\newcommand{\cO}{\ensuremath{\mathcal{O}}}

\newcommand{\C}{\ensuremath{\mathscr{C}}}

\newcommand{\normalJ}{\ensuremath{\mathscr{J}}}
\newcommand{\J}{\ensuremath{{\kern -0.4ex \mathscr{J}}}}
\newcommand{\Jalpha}{\J_{\kern -1ex \alpha}}

\newcommand{\sP}{\ensuremath{\mathscr{P}}}

\newcommand{\V}{\ensuremath{\mathscr{V}}}

\newcommand{\fJ}{\mathfrak{J}}
\newcommand{\fU}{\mathfrak{U}}

\newcommand{\fSet}{\mathfrak{Small}}

\newcommand{\Ja}{{\fJ_\alpha}}

\newcommand{\VCAT}{\ensuremath{\V\textnormal{-\text{CAT}}}}

\newcommand{\ALGCAT}{\ensuremath{\operatorname{\textnormal{\text{ALGCAT}}}}}
\newcommand{\ALGCATJ}{\ensuremath{\ALGCAT_{\kern -0.6ex \normalJ}}}
\newcommand{\ADA}{\ensuremath{\operatorname{\textnormal{\text{ADA}}}}}
\newcommand{\ADAJ}{\ensuremath{\ADA_{\kern -0.6ex \normalJ}}}

\newcommand{\TT}{\ensuremath{\mathbb{T}}}

\newcommand{\End}{\ensuremath{\operatorname{\textnormal{End}}}}

\newcommand{\stt}{\mathbin{\tilde{*}}}

\newcommand{\Th}{\ensuremath{\textnormal{Th}}}
\newcommand{\ThJ}{\ensuremath{\Th_{\kern -0.5ex \normalJ}}}

\newcommand{\otimesJ}[2]{\ensuremath{#1 \kern-0.15ex \otimes_{\kern -1ex \normalJ} \kern-0.5ex #2}}
\newcommand{\totimes}{\mathbin{\tilde{\otimes}}}
\newcommand{\totimesJ}[2]{\ensuremath{#1 \kern-0.15ex \totimes_{\kern -1ex \normalJ} \kern-0.5ex #2}}

\newcommand{\Vfp}{\ensuremath{\V_{\kern -0.5ex fp}}}
\newcommand{\TTperpJ}{\ensuremath{\TT^\perp_{\kern-.1ex\scriptscriptstyle\J}}}
\newcommand{\TTperpJwrt}[1]{\ensuremath{\TT^\perp_{\kern-.1ex\scriptscriptstyle\J#1}}}

\newcommand{\VCATJ}{\VCAT_{\kern -0.7ex \normalJ}}
\newcommand{\PhiJ}{\Phi_{\kern -0.8ex \normalJ}}
\newcommand{\MndJ}{\Mnd_{\kern -0.8ex \normalJ}}

\newcommand{\Set}{\ensuremath{\operatorname{\textnormal{\text{Set}}}}}

\newcommand{\SET}{\ensuremath{\operatorname{\textnormal{\text{SET}}}}}

\newcommand{\Mnd}{\ensuremath{\operatorname{\textnormal{\text{Mnd}}}}}

\newcommand{\Alg}[1]{\ensuremath{#1\kern -.5ex\operatorname{\textnormal{-Alg}}}}
\newcommand{\CAlg}[1]{\ensuremath{#1\kern -.5ex\operatorname{\textnormal{-\text{CAlg}}}}}
\newcommand{\CAlgPair}[2]{\ensuremath{(#1,#2)\kern -.5ex\operatorname{\textnormal{-CPair}}}}
\newcommand{\AlgPair}[2]{\ensuremath{(#1,#2)\kern -.5ex\operatorname{\textnormal{-Pair}}}}
\newcommand{\CinftyRing}{\ensuremath{C^\infty\kern -.5ex\operatorname{\textnormal{-\text{Ring}}}}}

\newcommand{\CRProf}{\ensuremath{\operatorname{\textnormal{\text{CRProf}}}}}
\newcommand{\CRProfJ}{\ensuremath{\CRProf_{\kern -0.6ex \normalJ}}}
\newcommand{\Algs}{\ensuremath{\operatorname{\textnormal{\text{Alg}}}}}
\newcommand{\AlgsJ}{\Algs_{\kern -0.6ex \normalJ}}

\newcommand{\BifoldAlg}{\ensuremath{\operatorname{\textnormal{\text{BAlg}}}}}
\newcommand{\BifoldAlgJ}{\BifoldAlg_{\kern -0.6ex \normalJ}}

\newcommand{\Adjn}[6]{\xymatrix {#1 \ar@/_0.5pc/[rr]_{#2}^(0.4){#4}^(0.6){#5}^{\top} & & #6 \ar@/_0.5pc/[ll]_{#3}}}
\newcommand{\Equiv}[6]{\xymatrix {#1 \ar@/_0.5pc/[rr]_{#2}^(0.4){#4}^(0.6){#5}^{\sim} & & #6 \ar@/_0.5pc/[ll]_{#3}}}
\newcommand{\EquivAlt}[6]{\xymatrix {#1 \ar@{<-}@/_0.5pc/[rr]_{#3}^(0.4){#4}^(0.6){#5}^{\sim} & & #6 \ar@{<-}@/_0.5pc/[ll]_{#2}}}
\newcommand{\AlgDual}[4]{\xymatrix {#1:#2 \ar@/_0.2pc/[r] & #3:#4 \ar@/_0.2pc/[l] }}

\newcommand{\lt}{\leqslant}

\newcommand{\pushoutcorner}{\ar@{}[dr]|(.3)\ulcorner}
\newcommand{\pullbackcorner}{\ar@{}[dr]|(.3)\lrcorner}

\newcommand{\cmt}[1]{}

\newcommand{\dn}{\mathop{\downarrow}}
\newcommand{\up}{\mathop{\uparrow}}

\begin{document}

\author{\normalsize  Rory B. B. Lucyshyn-Wright\thanks{The authors acknowledge the support of the Brandon University Research Committee (BURC).  The first author acknowledges the support of the Natural Sciences and Engineering Research Council of Canada (NSERC), [funding reference numbers RGPIN-2019-05274, RGPAS-2019-00087, DGECR-2019-00273].  Le premier auteur remercie le Conseil de recherches en sciences naturelles et en génie du Canada (CRSNG) de son soutien, [numéros de référence RGPIN-2019-05274, RGPAS-2019-00087, DGECR-2019-00273].}\\
\small Brandon University, Brandon, Manitoba, Canada \medskip \\ \normalsize Darian McLaren \\
\small University of Waterloo, Waterloo, Ontario, Canada\let\thefootnote\relax
}

\title{\large \textbf{Centralizer clones relative to a strong limit cardinal}}

\date{}

\maketitle

\abstract{
The notion of commutation of operations in universal algebra leads to the concept of \textit{centralizer clone} and gives rise to a well-known class of problems that we call \textit{centralizer problems}, in which one seeks to determine whether a given set of operations arises as a centralizer or, equivalently, coincides with its own double centralizer.  Centralizer clones and centralizer problems in universal algebra have been studied by several authors, with early contributions by Cohn, Kuznecov, Danil'\v{c}enko, and Harnau.  In this paper, we work within a generalized setting of infinitary universal algebra relative to a regular cardinal $\alpha$, thus allowing operations whose arities are sets of cardinality less than $\alpha$, and we study a notion of centralizer clone that is defined relative to $\alpha$.  In this setting, we establish several new characterizations of centralizer clones and double centralizer clones, with special attention to the case in which $\alpha$ is a strong limit cardinal, and we discuss how these results enable a novel method for treating centralizer problems.  We apply these results to establish positive solutions to finitary and infinitary centralizer problems for several specific classes of algebraic structures, including vector spaces, free actions of a group, and free actions of a free monoid.
}

\section{Introduction}\label{sec:intro}

The notion of a commuting pair of endomorphisms admits a well-known generalization in universal algebra.  Indeed, given pair of operations $f:A^n \rightarrow A$ and $g:A^m \rightarrow A$ on a set $A$, where $n$ and $m$ are finite cardinals, one says that $f$ \textit{commutes with} $g$, written $f \perp g$, if $$g\bigl(f(a_{11},...,a_{n1}),...,f(a_{1m},...,a_{nm})\bigr) = f(g(a_{11},...,a_{1m}),...,g\bigl(a_{n1},...,a_{nm})\bigr)$$
for every $n \times m$-matrix $a \in A^{n\times m}$ with entries in $A$.  Just as the notion of commutation in algebra leads to the familiar notion of centralizer in groups and rings, this notion of commuting pair of operations gives rise to a notion of centralizer in universal algebra: If $G$ is a set of (finitary) operations on a given set $A$, then the \textit{centralizer} of $G$ is the set $G^\perp$ consisting of all operations $f$ on $A$ such that $f \perp g$ for all $g \in G$ \cite[III.3]{Cohn}.  A set of operations $F$ on $A$ is the centralizer $G^\perp$ of some set of operations $G$ if and only if $F$ is its own double-centralizer $F^{\perp\perp}$.  Thus there arises a well-known class of problems, which we call \textit{centralizer problems}, that one may pose for various sets of finitary operations $F$ on various sets $A$: \textit{Determine whether $F$ is a centralizer}.

We refer the reader to \cite[\S 1]{BehrVar} for an account of the history of work on centralizers and centralizer problems in universal algebra and for comprehensive references to the early contributions of Cohn, Kuznecov, Danil'\v{c}enko, and Harnau as well as references to some more recent contributions (e.g. \cite{Hermann,MachRos}); also see \cite{BurrisWillard,Burris,Snow,SichTrn:All,SichTrn:Inf,Lu:CvxAffCmt,Lu:Cmt}.

For a set of operations $F$ on a set $A$ to be a centralizer, one basic necessary condition is that $F$ must be a (concrete) \textit{clone}, meaning that $F$ must contain the projection maps $\pi_i:A^n \rightarrow A$ and $F$ must be closed under the formation of composites $g \circ (f_1,...,f_m):A^n \rightarrow A$ of operations $g:A^m \rightarrow A$ and $f_1,...,f_m:A^n \rightarrow A$; see \cite[III.3]{Cohn}.  Clones play an important role in universal algebra, as is evident in \cite{Szend} for example, beginning with the fact that every algebra $A$ (in the universal-algebraic sense) determines a clone, consisting of the \textit{derived operations} of $A$.  In detail, an \textit{algebra} (or \textit{$\Sigma$-algebra}) in this sense is a set $A$ equipped with a family of operations $\sigma^A:A^{n_\sigma} \rightarrow A$ indexed by the elements $\sigma \in \Sigma$ of a \textit{signature} or \textit{similarity type} $\Sigma$, meaning that $\Sigma$ is a set with an assignment to each $\sigma \in \Sigma$ a finite cardinal $n_\sigma$ called the arity of $\sigma$.  The derived operations on $A$ are then the elements of the smallest clone that contains the basic operations $\sigma^A$ $(\sigma \in \Sigma)$.

Following S{\l}omi\'{n}ski \cite{Slo}, it is also possible to consider \textit{infinitary} operations, i.e. maps $f:A^J \rightarrow A$ where $J$ is an arbitrary cardinal or an arbitrary set, and consequently the notions of algebra, of clone, and of centralizer clone admit generalizations to the infinitary setting.  However, almost all existing works on centralizer clones concern \textit{finitary} clones, in which the arities $n$ are finite cardinals; exceptions include \cite{SichTrn:Inf,Lu:Cmt}, both of which are concerned with \textit{abstract} infinitary clones or \textit{Lawvere theories}.  Moreover, much of the existing work on centralizer clones concerns the case where the set $A$ itself is finite; this is partly because the works of Kuznecov, Danil'\v{c}enko, and Harnau on centralizer clones arose within the field of \textit{multi-valued logic} and are based in part on a result of Kuznecov that describes centralizer clones on finite sets $A$ in logical terms via \textit{primitive positively definable relations}; see \cite{BehrVar} for a survey.

In the present paper, we consider both infinitary and finitary centralizer clones, and we develop novel techniques and results for centralizer clones that are based on the use of infinitary clones but are applicable both to infinitary and finitary clones.

Before describing our results, let us briefly indicate the setting in which we work.  We fix a regular cardinal $\alpha$, and we consider operations $f:A^J \rightarrow A$ whose arities $J$ are sets of cardinality less than $\alpha$.  For added flexibility, but no essential gain in generality, we allow the arities $J$ to be drawn from a specified set $\Ja$ that contains all cardinals $< \alpha$ and consists only of sets of cardinality $< \alpha$.  The notion of (finitary) clone\footnote{In this paper, we allow nullary operations as in \cite[III.3]{Cohn}, while various works do not (e.g. \cite{Szend}).} as in \cite[III.3]{Cohn} then generalizes to the notion of (concrete) \textit{$\Ja$-ary clone}, and the notion of (finitary) centralizer clone as in \cite[III.3]{Cohn} generalizes to the notion of \textit{$\Ja$-ary centralizer clone} (\ref{defn:centralizer}).  Of course, finitary clones are then recovered by taking $\alpha = \aleph_0$ and taking $\Ja$ to consist of the finite cardinals.  On the other hand, another equally interesting but less familiar case is obtained by taking $\alpha$ to be an \textit{inaccessible cardinal} $\kappa$ (see e.g. \cite{Jech}) and taking $\Ja$ to consist of the sets of rank less than $\kappa$.  As we recall in \ref{para:grothu}, an equivalent way of describing this case is to let $\alpha = \kappa$ be the cardinality of a given \textit{Grothendieck universe} $\fU$ (\cite[I(1)]{Gab}, \cite[I]{Gr:SGA4-1}, \cite{Wil:GrothU}) and let $\Ja = \fU$.  Explicitly, a Grothendieck universe is a set $\fU$ whose elements are sets and which is closed under various set-theoretic constructions (which we recall in \ref{para:grothu}), in such a way that $\fU$ itself is a model of Zermelo-Fraenkel set theory.  Indeed, a common approach to the set-theoretic foundations of category theory and algebraic geometry is to take the view that the `usual' practice of algebra, analysis, and topology takes place in a given Grothendieck universe $\fU$, whose elements one calls \textit{$\fU$-small sets} or simply \textit{small sets}.  Thus, in this paper we let $\fSet$ denote a given Grothendieck universe $\fU$ whose elements we call small sets; equivalently, we take $\fSet$ to be the set of all sets of rank less than an inaccessible cardinal $\kappa$ whose existence we postulate.  By taking $\alpha = \kappa$ and $\Ja = \fSet$, we arrive at the notion of (concrete) \textit{$\fSet$-ary clone}, i.e. a clone consisting of operations $A^J \rightarrow A$ whose arities $J$ are allowed to be arbitrary small sets, and in particular we can form the \textit{$\fSet$-ary centralizer clone} of a given set of operations with small arities.  Let us mention in passing that, from a categorical point of view, the importance of (concrete) $\fSet$-ary clones lies partly in the fact that \textit{abstract} $\fSet$-ary clones may be described equivalently as arbitrary \textit{monads} on the category $\Set$ of small sets (as has been known since \cite{Lin:Eq}), which provide a very general setting for the study of categories of algebraic structures for which free algebras exist.  While in general we assume simply that $\alpha \lt \kappa$ is a regular cardinal, some of our main results require the additional assumption that $\alpha$ is a \textit{strong limit} cardinal, which is true if $\alpha = \aleph_0$ or if $\alpha = \kappa$ is inaccessible.

Our main results provide powerful new methods for solving centralizer problems in both the $\fSet$-ary context and also the traditional finitary context.  In brief, the main results of this paper comprise (1) theorems characterizing the $\Ja$-ary centralizer clone $G^\perp$ of a set of $\Ja$-ary operations $G$ on a small set $A$ (Theorems \ref{thm:charn_cc} and \ref{thm:cc_thm_for_slc}), (2) a theorem characterizing the $\Ja$-ary double centralizer clone $(\Sigma^A)^{\perp\perp}$ of an algebra $A$ of cardinality $<\alpha$ (Theorem \ref{thm:dbl_cc}), where we write $\Sigma^A$ for the set of all basic operations carried by the $\Sigma$-algebra $A$, with corollaries for the $\fSet$-ary and finitary cases (\ref{thm:fset_ary_dbl_centralizer}, \ref{thm:dbl_cc_tech_for_finitary}), and (3) several applications to specific examples of classes of algebras, including vector spaces (\S\ref{sec:vec}), actions of groups (\S\ref{sec:gset}), and actions of free monoids (\S\ref{sec:free_mon_actions}), with solutions to the $\fSet$-ary and finitary centralizer problems for various such structures.

Looking in more detail at the methods and results in this paper, one of the key ideas is that for any operation $f:A^J \rightarrow A$ on a set $A$, if the sets $A$ and $J$ both have cardinality less than a given strong limit cardinal $\alpha$, such as $\aleph_0$ or $\kappa$ as above, then the set $A^J$ also has cardinality less than $\alpha$, and we may consider operations $h:A^{(A^J)} \rightarrow A$ of arity $A^J$ and apply such operations $h$ to $f \in A^{(A^J)}$ itself.  This idea appears first in Theorems \ref{thm:charn_cc} and \ref{thm:cc_thm_for_slc}, where we establish new characterizations of the $\Ja$-ary centralizer clone that are at the technical core of the paper and enable several further results later in the paper.  In detail, we show in Theorem \ref{thm:cc_thm_for_slc} that if $G$ is a $\Ja$-ary clone on a set $A$ of cardinality $<\alpha$, for a regular strong limit cardinal $\alpha$, then the $\Ja$-ary centralizer $G^\perp$ of $G$ consists of precisely those $\Ja$-ary operations $f:A^J \rightarrow A$ such that
\begin{equation}\label{eq:intro_hf_eqn}h(f) = f(h \circ u_J)\end{equation}
for every operation of the form $h:A^{(A^J)} \rightarrow A$ in $G$, where $u_J:J \rightarrow A^{(A^J)}$ is the map that sends each $j \in J$ to the projection map $\pi_j:A^J \rightarrow A$. This result provides a substantial reduction in complexity relative to the definition of the centralizer clone and applies in particular to both $\fSet$-ary clones on a small set $A$ (\ref{thm:charn_cc_cor_set}) and finitary clones on a finite set $A$ (\ref{thm:charn_cc_cor_fin}).  Furthermore, this theorem is a corollary to a more general result that is applicable for an arbitrary set $G$ of $\Ja$-ary operations on an arbitrary small set $A$, for an arbitrary regular cardinal $\alpha$, namely Theorem \ref{thm:charn_cc}, where we show that the $\Ja$-ary centralizer $G^\perp$ of $G$ consists of precisely those $\Ja$-ary operations $f:A^J \rightarrow A$ that satisfy \eqref{eq:intro_hf_eqn} for every operation of the form $h:A^{(A^J)} \rightarrow A$ that lies within the $\fSet$-ary clone $\langle G \rangle$ generated by $G$ (i.e. the smallest $\fSet$-ary clone that contains $G$).  As a further part of Theorem \ref{thm:charn_cc}, we show also that the latter condition on $f$ is equivalent to the requirement that $f$ commute with every operation of the form $h:A^{(A^J)} \rightarrow A$ in $\langle G\rangle$.  Thus, while the definition of membership of $f:A^J \rightarrow A$ in the centralizer $G^\perp$ involves quantification over all arities $K \in \Ja$ and all $J \times K$-matrices $a \in A^{J \times K}$, Theorem \ref{thm:charn_cc} allows us to restrict our attention to just the (small) arity $A^J$ and, in fact, just the \textit{evaluation matrix} $e = (a_j)_{(j,a) \in J \times A^J} \in A^{J \times A^J}$.

These results have their greatest impact in proving our central Theorem \ref{thm:dbl_cc} on \textit{double} centralizers, in which we suppose that $\alpha$ is a regular strong limit cardinal and we show that if $A$ is an algebra of cardinality $<\alpha$ over a $\Ja$-ary signature $\Sigma$, then its $\Ja$-ary double centralizer clone $(\Sigma^A)^{\perp\perp}$ consists of precisely those $\Ja$-ary operations $f:A^J \rightarrow A$ that satisfy \eqref{eq:intro_hf_eqn} for every \textit{$\Sigma$-homomorphism} $h:A^{(A^J)} \rightarrow A$.  In particular, we thus obtain particularly useful characterizations of (1) the $\fSet$-ary double centralizer clone of a small algebra $A$ (in Corollary \ref{thm:fset_ary_dbl_centralizer}) and (2) the finitary double centralizer clone of a finite algebra $A$ (in Corollary \ref{thm:finitary_dbl_centralizer}).

These results provide a useful method for establishing positive solutions to various specific centralizer problems, as follows.  In the situation of the preceding paragraph, the \textit{$\Ja$-ary centralizer problem} for the algebra $A$ is the question of whether the $\Ja$-ary double centralizer clone $(\Sigma^A)^{\perp\perp}$ is equal to the clone of $\Ja$-ary derived operations of $A$, which we denote by $\langle \Sigma^A \rangle_\alpha$.  The latter is always contained in the former, so Theorem \ref{thm:dbl_cc} entails that a positive solution to the $\Ja$-ary centralizer problem for $A$ will be obtained as soon as we can show that for every $\Ja$-ary operation $f:A^J \rightarrow A$, if $f \notin \langle \Sigma^A\rangle_\alpha$ then there exists some $\Sigma$-homomorphism $h:A^{(A^J)} \rightarrow A$ such that $h(f) \neq f(h \circ u_J)$; see Corollary \ref{thm:method_for_showing_alg_sat}.

We apply these results and methods to provide positive solutions to $\fSet$-ary and finitary centralizer problems for various specific algebras in Sections \ref{sec:vec}, \ref{sec:gset}, and \ref{sec:free_mon_actions}, as we now summarize.

In Section \ref{sec:vec}, we begin with a basic and illustrative class of examples, showing in Theorem \ref{thm:vec_sp_thm} that if $V$ is a small, nonzero $K$-vector space over a small field $K$, then the $\fSet$-ary double centralizer clone of $V$ is equal to the clone of $\fSet$-ary derived operations of $V$.  In Theorem \ref{thm:finite_field_vec} we similarly show that if $K$ is a \textit{finite} field and $V$ is a finite-dimensional, nonzero $K$-vector space, then the finitary double centralizer clone of $V$ is equal to the clone of finitary derived operations of $V$.

In Section \ref{sec:gset}, we treat $\fSet$-ary and finitary centralizer problems for free actions of a group $G$ (also called free $G$-sets).  In Theorem \ref{thm:free_gset} we show that if $A$ is a small and free $G$-set with at least two elements, where $G$ is a small group, then the $\fSet$-ary double centralizer clone of $A$ is equal to the clone of $\fSet$-ary derived operations of $A$.  In Theorem \ref{thm:finite_gset}, we similarly show that if $G$ is a finite group and $A$ is a finite free $G$-set with at least two elements, then the finitary double centralizer clone of $A$ is equal to the clone of finitary derived operations of $A$.

In Section \ref{sec:free_mon_actions}, we treat $\fSet$-ary centralizer problems for free actions of a free monoid $S^*$ generated by a set $S$.  In Theorem \ref{thm:free_sstar_set}, we show that if $A$ is a small and free $S^*$-set with at least two elements, then the $\fSet$-ary double centralizer clone of $A$ is equal to the clone of $\fSet$-ary derived operations of $A$.

\medskip

In addition to assuming that the reader is familiar with set-theoretic concepts such as ordinals, cardinals, and Zermelo-Fraenkel set theory, we shall make some light use of basic concepts of category theory (e.g. categories, functors, products, powers, and coproducts, etcetera).  We review the notions of rank, regular cardinal, strong limit cardinal, and inaccessible cardinal in \ref{para:set_theory}.  As mentioned above, we also make use of the concept of Grothendieck universe that is frequently used in the foundations of category theory, and in \ref{para:grothu} we review this concept from a set-theoretic perspective, including its intimate relationship to inaccessible cardinals.

\section{Clones large and small}

In this section, we first discuss the set-theoretic concepts and given data with which we work throughout the paper, and then we discuss certain concepts of infinitary universal algebra relative to a regular cardinal $\alpha$, including the notion of \textit{$\Ja$-ary clone}.

\begin{para}\label{para:set_theory}
Let us begin by reviewing some well-known set-theoretic details; see, e.g. \cite{Jech}.  Let us recall that an infinite cardinal $\alpha$ is \textbf{regular} if $\sum_{i \in \gamma} \beta_i < \alpha$ for every family of cardinals $\beta_i < \alpha$ $(i \in \gamma)$ indexed by a cardinal $\gamma < \alpha$.  For example, the first infinite cardinal $\aleph_0$ is regular.  A \textbf{strong limit cardinal} is a nonzero cardinal $\alpha$ such that $2^\beta < \alpha$ for every cardinal $\beta < \alpha$ (equivalently, such that $\gamma^\beta < \alpha$ for every pair of cardinals $\beta,\gamma$ with $\beta < \alpha$ and $\gamma < \alpha$).  For example, $\aleph_0$ is a strong limit cardinal.  An \textbf{inaccessible cardinal} is a strong limit cardinal that is uncountable and regular.

By transfinite induction we may associate to each ordinal $\tau$ a set $\mathfrak{V}_\tau$ by declaring that (1) $\mathfrak{V}_0 = \emptyset$, (2) $\mathfrak{V}_{\tau+1}$ is the powerset $\sP(\mathfrak{V}_\tau)$ of $\mathfrak{V}_\tau$, and (3) $\mathfrak{V}_\tau = \bigcup_{\sigma < \tau} \mathfrak{V}_\sigma$ if $\tau$ is a limit ordinal.  It follows that if $\sigma$ and $\tau$ are ordinals with $\sigma \leq \tau$ then $\mathfrak{V}_\sigma \subseteq \mathfrak{V}_\tau$.  Given an ordinal $\sigma$, we say that a set $X$ has \textbf{rank} $\sigma$ if $X \in \mathfrak{V}_{\sigma+1}$ and $X \notin \mathfrak{V}_\sigma$ (equivalently if $X \in \mathfrak{V}_{\sigma+1}$ and $\sigma$ is minimal with this property).  The axiom of regularity in Zermelo-Fraenkel set theory entails that every set $X$ is an element of some $\mathfrak{V}_\tau$ and so has rank $\sigma$ for a unique ordinal $\sigma$ (\textbf{the rank of $X$}).  Hence, $\mathfrak{V}_\tau$ is \textbf{the set of all sets of rank less than $\tau$}.
\end{para}

\begin{para}\label{para:grothu}
A \textbf{Grothendieck universe} (\cite[I(1)]{Gab}, \cite[I]{Gr:SGA4-1}, \cite{Wil:GrothU}) is a set $\fU$ with the following properties\footnote{Some authors omit (4) and instead simply require that $\fU \neq \emptyset$, thus treating (4) as a further axiom that can be added.}: (1) Each element $X$ of $\fU$ is a subset of $\fU$; (2) for every $X \in \fU$, the powerset $\sP(X)$ and the singleton set $\{X\}$ are both elements of $\fU$, i.e. $\sP(X) \in \fU$ and $\{X\} \in \fU$; (3) $\bigcup_{i \in I} X_i \in \fU$ for every family $(X_i)_{i \in I}$ that consists of sets $X_i \in \fU$ and is indexed by a set $I \in \fU$; (4) $\fU$ contains an infinite set.

By \cite{Wil:GrothU}, a set $\fU$ is a Grothendieck universe if and only if $\fU = \mathfrak{V}_\kappa$ for some inaccessible cardinal $\kappa$, with the notation of \ref{para:set_theory}.  If $\fU$ is a Grothendieck universe, then $\fU$ is itself a model of Zermelo-Fraenkel set theory, under the usual membership relation $\in$, so that $\fU$ is actually closed under all the usual set-theoretic constructions (e.g. products, powers, disjoint unions, etcetera).  Thus, the key idea is that a Grothendieck universe $\fU$ provides a set-theoretic environment in which all of `ordinary mathematics' can be performed.
\end{para}

\begin{para}
Throughout this paper, we fix a Grothendieck universe $\fSet$ whose elements we call \textbf{small sets}, and we write $\Set$ to denote the category of small sets (in which the morphisms are all the functions between small sets), while we write $\SET$ for the category of all sets, calling a set \textbf{large} if it is not in $\fSet$.  In view of \ref{para:grothu}, it is equivalent to fix an inaccessible cardinal $\kappa$ and let $\fSet$ consist of all sets of rank less than $\kappa$, so that $\fSet = \mathfrak{V}_\kappa$.  By \cite[Exercise 6.3]{Jech}, $\kappa$ is then the cardinality of $\fSet$, so we call $\kappa$ the \textbf{cardinality of the universe of small sets, $\fSet$}, and denote it by $\kappa(\fSet)$.  Practically, the idea is that $\fSet$ (or $\Set$) is the mathematician's usual working universe of sets, while $\SET$ is a broader set-theoretic environment.
\end{para}

\begin{para}\label{para:ar}
Throughout the paper, we fix a (possibly large) regular cardinal $\alpha \leq \kappa(\fSet)$.  We say that a set $J$ is \textbf{$\alpha$-small} if it is small and the cardinality of $J$ is less than $\alpha$.  We write $\fSet_\alpha$ to denote the (possibly large) set of all $\alpha$-small sets, and we write $\mathfrak{Card}_\alpha$ to denote the set of all $\alpha$-small cardinals.  Note that since $\alpha$ is regular, if $(J_i)_{i \in I}$ is a family of $\alpha$-small sets $J_i \in \fSet_\alpha$ indexed by an $\alpha$-small set $I \in \fSet_\alpha$, then the disjoint union $\coprod_{i \in I} J_i$ is $\alpha$-small, i.e. $\coprod_{i \in I} J_i \in \fSet_\alpha$.  In particular, if $J,K \in \fSet_\alpha$ then the product $J \times K$ is a disjoint union $\coprod_{j \in J} K$ and so lies in $\fSet_\alpha$.  If $\alpha$ is also a strong limit cardinal, then $K^J \in \fSet_\alpha$ for all $J,K \in \fSet_\alpha$.

Our intention is to employ either the $\alpha$-small cardinals or the $\alpha$-small sets in general as \textit{arities} for infinitary algebra, so, to allow some flexibility in this regard, we fix a set $\Ja$ with
$$\mathfrak{Card}_\alpha \subseteq \Ja \subseteq \fSet_\alpha\;.$$
We call the sets $J \in \Ja$ the \textbf{$\alpha$-small arities}.  The full subcategories $\textnormal{Card}_\alpha$, $\Set_{\Ja}$, and $\Set_\alpha$ of $\Set$ consisting of the sets in $\mathfrak{Card}_\alpha$, $\Ja$, and $\fSet_\alpha$, respectively, are all equivalent, and at times we even take the liberty of writing as if every $\alpha$-small set $J$ lies in $\Ja$ (while strictly speaking we must choose a set in $\Ja$ equinumerous to $J$, such as the cardinality of $J$).  In particular, given $I,J,K,J_i \in \Ja$ $(i \in I)$, we shall identify the $\alpha$-small sets $\coprod_{i \in I} J_i$ and $J \times K$ with associated sets in $\Ja$, and similarly for $K^J$ if $\alpha$ is strong limit cardinal.
\end{para}

\begin{exa}[\textbf{Finite arities, finite cardinals}]
We may take $\alpha = \aleph_0$ and take either $\Ja = \fSet_{\aleph_0}$ or $\Ja = \mathfrak{Card}_{\aleph_0} = \aleph_0$.
\end{exa}

\begin{exa}[\textbf{Arbitrary small arities, small cardinals}]\label{exa:small_arities}
We may take $\alpha = \kappa(\fSet)$ to be the cardinality of the universe of small sets, $\fSet$, and take either $\Ja = \fSet$, so that $\Ja$ consists of \textit{all} small sets, or take $\Ja = \mathfrak{Card} := \mathfrak{Card}_{\kappa(\fSet)}$, so that $\Ja$ consists of all small cardinals.
\end{exa}

Note that, in the preceding two examples, $\alpha$ is not only a regular cardinal but is also a strong limit cardinal.

\begin{para}
An \textbf{($\Ja$-ary) signature} (or \textit{$\Ja$-ary similarity type}, or \textit{$\Ja$-graded set}) is a (possibly large) set $\Sigma$ equipped with a map $\mathsf{arity}:\Sigma \rightarrow \Ja$.  The elements $\sigma \in \Sigma$ are called \textit{operation symbols}, and we call the set $\mathsf{arity}(\sigma)$ the \textit{arity} of $\sigma$.  For each $J \in \Ja$ we write $\Sigma(J)$ to denote the fibre $\mathsf{arity}^{-1}(J) = \{\sigma \in \Sigma \mid \mathsf{arity}(\sigma) = J\}$. Given signatures $\Sigma$ and $\Gamma$, a \textbf{morphism of signatures} $\phi:\Sigma \rightarrow \Gamma$ is a map that commutes with the arity maps, i.e. with $\mathsf{arity}(\phi(\sigma)) = \mathsf{arity}(\sigma)$ for all $\sigma \in \Sigma$.  A signature is equivalently\footnote{up to an equivalence of categories; precisely, the above passages describe an equivalence of categories $\SET\slash\Ja \simeq \SET^{\Ja}$ between the slice category over $\Ja$ in $\SET$ and the power $\SET^\Ja$ of $\SET$ by $\Ja$.} given by a $\Ja$-indexed family of sets $\Sigma(J)$ $(J \in \Ja)$, whose associated signature in the above sense is the disjoint union $\Sigma = \coprod_{J \in \Ja} \Sigma(J)$, while a morphism of signatures $\phi:\Sigma \rightarrow \Gamma$ is equivalently given by a family of maps $\phi_J:\Sigma(J) \rightarrow \Gamma(J)$ $(J \in \Ja)$.  We shall employ these two notions of signature interchangeably throughout.
\end{para}

\begin{exa}
Upon taking $\alpha = \kappa(\fSet)$ and $\Ja = \fSet$, we find that a \textit{$\fSet$-ary signature} is a (possibly large) set $\Sigma$ equipped with a map $\mathsf{arity}:\Sigma \rightarrow \fSet$.
\end{exa}

\begin{para}\label{para:ops}
Given a set $A$, an \textit{operation} on $A$ is a map $f:A^J \rightarrow A$ for a specified set $J$ called the \textit{arity} of the operation $f$, and we say that $f$ is a \textbf{$\Ja$-ary operation} if $J \in \Ja$.  We write $\cO_A^{\,\alpha}$ to denote the signature consisting of all $\Ja$-ary operations on $A$, so that $\cO_A^{\,\alpha}(J) = A^{(A^J)}$ for all $J \in \Ja$.  Hence $\cO_A^{\,\alpha}$ is the disjoint union $\coprod_{J \in \Ja}A^{(A^J)}$ equipped with its evident arity map valued in $\Ja$.  We write simply $\cO_A$ to denote the $\fSet$-ary signature $\cO_A^{\,\kappa(\fSet)}$, so that $\cO_A(X) = A^{(A^X)}$ for every small set $X$.

We shall also make use of the fact that $\cO_A(X)$ is covariantly functorial in $X \in \Set$.  Explicitly, given any map $\zeta:X \rightarrow Z$ between small sets $X$ and $Z$, we obtain a map
$$\zeta_* = A^{A^\zeta}\;:\;\cO_A(X) = A^{(A^X)} \rightarrow \cO_A(Z) = A^{(A^Z)}$$
that sends each operation $f:A^X \rightarrow A$ to the \textbf{pushforward} $\zeta_*(f) = f \circ A^\zeta:A^Z \rightarrow A$, given by $\zeta_*(f)(a) = f(a \circ \zeta)$ $(a \in A^Z)$.
\end{para}

\begin{para}
Given any signature $\Sigma$, a \textbf{$\Sigma$-algebra} is a set $A$ equipped with a morphism of signatures $\Sigma \rightarrow \cO_A^{\,\alpha}$, whose value at each operation symbol $\sigma \in \Sigma(J)$ $(J \in \Ja)$ we write as $\sigma^A:A^J \rightarrow A$ and call a \textit{basic $\Sigma$-operation}.  $\Sigma$-algebras are the objects of a category $\Sigma\textnormal{-ALG}$ in which the morphisms are \textbf{$\Sigma$-homomorphisms}, i.e. maps that preserve the basic $\Sigma$-operations.  A \textbf{($\Ja$-ary) algebra} is a pair $(\Sigma,A)$ consisting of a signature $\Sigma$ and a $\Sigma$-algebra $A$.  A \textbf{$\Sigma$-subalgebra} of $A$ is then a subset $B$ of $A$ that is \textit{closed under} each basic $\Sigma$-operation $\sigma^A:A^J \rightarrow A$ $(J \in \Ja, \sigma \in \Sigma(J))$, meaning that $\sigma^A(b) \in B$ for every every $b \in B^J$; $B$ then acquires the structure of a $\Sigma$-algebra.  Given a subset $S$ of a $\Sigma$-algebra $A$, there is a smallest $\Sigma$-subalgebra $B$ of $A$ containing $S$, namely the intersection of all such, which we call the \textbf{$\Sigma$-subalgebra of $A$ generated by $S$}.
\end{para}

\begin{para}\label{para:set_of_jary_ops}
Given a set $A$, a \textbf{set of $\Ja$-ary operations on $A$} is a subset $F$ of $\cO_A^{\,\alpha}$.  Any set of $\Ja$-ary operations $F$ on $A$ carries the structure of a signature in such a way that the inclusion $F \hookrightarrow \cO^{\,\alpha}_A$ is a morphism of signatures.  The set $A$ then carries the structure of an $F$-algebra for this signature $F$, so that $(F,A)$ is an algebra.  Conversely, given any algebra $(\Sigma,A)$, the basic $\Sigma$-operations on $A$ constitute a set of $\Ja$-ary operations $\Sigma^A := \{\sigma^A \mid \sigma \in \Sigma\}$.
\end{para}

\begin{para}\label{para:clone}
Given a set $A$, a \textbf{(concrete) $\Ja$-ary clone} on $A$, or simply a \textbf{clone} on $A$, is a set $F$ of $\Ja$-ary operations on $A$ with the following properties: (1) $F$ contains the projection maps $\pi_j:A^J \rightarrow A$ with $J \in \Ja$ and $j \in J$, and (2) for every pair of sets $J,K \in \Ja$, every operation $f:A^K \rightarrow A$ in $F$, and every $K$-indexed family of operations $g_k:A^J \rightarrow A$ $(k \in K)$ in $F$, the composite
$$f \circ (g_k)_{k \in K} \;=\; \bigl(A^J \xrightarrow{(g_k)_{k \in K}} A^K \xrightarrow{f} A\bigr)$$
lies in $F$, where $(g_k)_{k \in K}$ denotes the map induced by the operations $g_k$.
\end{para}

\begin{rem}\label{rem:constants}
$\Ja$-ary operations on $A$ as defined in \ref{para:ops} include \textit{nullary} operations, i.e.~operations of arity $0 = \emptyset$, which are equivalently \textit{constants}, i.e.~elements of $A$.  Hence a $\Ja$-ary clone in the sense of \ref{para:clone} may contain nullary operations.
\end{rem}

\begin{exa}\label{exa:fin_cl}
The traditional notion of clone in \cite[III.3]{Cohn} is recovered in the case where $\alpha = \aleph_0$ and where $\Ja = \mathfrak{Card}_{\aleph_0} = \aleph_0$ is the set of all finite cardinals.  We therefore refer to $\aleph_0$-ary clones as \textbf{finitary clones}.  We allow nullary operations as in \cite[III.3]{Cohn}, in view of \ref{rem:constants}, while various works do not, including \cite{Szend}; see \cite{Behr:Nullary} for a discussion of the relationship between the resulting two notions of finitary clone.
\end{exa}

\begin{para}\label{para:clone_gend_by}
Given any set of $\Ja$-ary operations $F$ on a set $A$, since $A$ is canonically an $F$-algebra \pref{para:set_of_jary_ops}, we know that each power $A^X$ is an $F$-algebra (under the pointwise operations, making $A^X$ a power in the category of $\Sigma$-algebras).  In particular, the set $\cO_A(J) = A^{(A^J)}$ is an $F$-algebra for each $J \in \Ja$, with basic $F$-operations $f^{\cO_A(J)}:\cO_A(J)^K \rightarrow \cO_A(J)$ $(K \in \Ja, f \in F(K))$ given by $(g_k)_{k \in K} \mapsto f \circ (g_k)_{k \in K}$.  It is then immediate that $F$ is a clone on $A$ iff for each $J \in \Ja$ the set $F(J)$ contains the projections $\pi_j$ $(j \in J)$ and is an $F$-subalgebra of $\cO_A(J)$.

Given any set $F$ of $\Ja$-ary operations on $A$, there is a smallest clone $\langle F \rangle_\alpha$ on $A$ with $F \subseteq \langle F \rangle_\alpha$, namely the intersection of all clones $G$ on $A$ with $F \subseteq G$, and we call $\langle F \rangle_\alpha$ the \textbf{($\Ja$-ary) clone generated by $F$}.  Given any subset $B$ of $A$, the set $\{ f \in \cO_A^{\,\alpha} \mid \text{$B$ is closed under $f$}\}$ is clearly a clone on $A$.  Hence, given any set $F$ of $\Ja$-ary operations on $A$, a subset $B$ of $A$ is an $F$-subalgebra of $A$ if and only if $B$ is an $\langle F\rangle_\alpha$-subalgebra of $A$.  Using this, we can prove the following basic result:
\end{para}

\begin{prop}\label{thm:charn_clone_gend_by}
Let $F$ be a set of $\Ja$-ary operations on a set $A$.  For each $J \in \Ja$, $\langle F \rangle_\alpha(J)$ is the $F$-subalgebra of $\cO_A(J)$ generated by the projections $\pi_j:A^J \rightarrow A$ $(j \in J)$.
\end{prop}
\begin{proof}
Let $G$ be the set of $\Ja$-ary operations on $A$ defined by declaring for each $J \in \Ja$ that $G(J)$ is the $F$-subalgebra of $\cO_A(J)$ generated by $\Pi_J = \{\pi_j \mid j \in J\}$.  It is then straightforward to show that $F \subseteq G$.  For each $J \in \Ja$, $\langle F\rangle_\alpha(J)$ contains $\Pi_J$ and is an $F$-subalgebra of $\cO_A(J)$ (as it is an $\langle F \rangle_\alpha$-subalgebra, by \ref{para:clone_gend_by}), so $G(J) \subseteq \langle F\rangle_\alpha(J)$.  Hence $G \subseteq \langle F \rangle_\alpha$, so since each $G(J)$ is an $\langle F\rangle_\alpha$-subalgebra of $\cO_A(J)$ by \ref{para:clone_gend_by}, we deduce that $G(J)$ is a $G$-subalgebra of $\cO_A(J)$.  Therefore $G$ is a clone on $A$, by \ref{para:clone_gend_by}, so since $F \subseteq G$ we find that $\langle F \rangle_\alpha \subseteq G$.
\end{proof}

\begin{para}
Given any algebra $(\Sigma,A)$, we may form the clone $\langle \Sigma^A \rangle_\alpha$ generated by the set of basic $\Sigma$-operations $\Sigma^A$ on $A$.  We call the elements of $\langle \Sigma^A \rangle_\alpha$ the \textbf{$(\Ja$-ary) derived $\Sigma$-operations} on $A$, and we call $\langle \Sigma^A \rangle_\alpha$ the \textbf{($\Ja$-ary) clone of derived operations} of $(\Sigma,A)$.  In particular, if $F$ is any set of $\Ja$-ary operations on a set $A$, then $(F,A)$ is an algebra \pref{para:set_of_jary_ops} whose clone of derived operations is the clone $\langle F \rangle_\alpha$ generated by $F$.
\end{para}

\section{Centralizer clones large and small}

\begin{para}\label{para:kr_pr}
Let $f:A^X \rightarrow A$ and $g:A^Y \rightarrow A$ be operations on a set $A$, where $X,Y$ are small sets.  Then the \textbf{first and second Kronecker products}\footnote{This terminology is derived from \cite{Lu:CvxAffCmt,Lu:Cmt}.} of $f$ and $g$ are the operations $f * g, f \stt g:A^{X \times Y} \rightarrow A$ defined as the composites
$$f * g = \bigl(A^{X \times Y} \xrightarrow{\sim} (A^X)^Y \xrightarrow{f^Y} A^Y \xrightarrow{g} A\bigr)$$
$$f\stt g = \bigl(A^{X \times Y} \xrightarrow{\sim} (A^Y)^X \xrightarrow{g^X} A^X \xrightarrow{f} A\bigr)$$
where $\sim$ denotes the canonical bijection.  Explicitly, given any $X\times Y$-matrix $a \in A^{X \times Y}$, if we let $\acute{a}:X \rightarrow A^Y$ and $\grave{a}:Y \rightarrow A^X$ be the maps given by $\acute{a}(x)(y) = a_{xy} = \grave{a}(y)(x)$ $(x \in X, y \in Y)$, then
$$(f*g)(a) = g(f \circ \grave{a})\;\;\text{and}\;\;(f\stt g)(a) = f(g \circ \acute{a})\;.$$
We say that $f$ \textbf{commutes with} $g$, written $f \perp g$, if $f*g = f \stt g$.  Clearly $f \perp g$ iff $g \perp f$.
\end{para}

\begin{defn}\label{defn:centralizer}
Given a set of $\Ja$-ary operations $G$ on a set $A$, we write
$$G^\perp = G^{\perp(\Ja)} = \{f \in \cO_A^{\,\alpha} \mid \forall g \in G\;:\; f \perp g\}$$
and we call $G^\perp$ the \textbf{$\Ja$-ary centralizer} (or \textbf{commutant}) of $G$.  We denote the $\Ja$-ary centralizer of $G^\perp$ by $G^{\perp\perp}$ or $G^{\perp\perp(\Ja)}$ and call it the \textbf{$\Ja$-ary double centralizer} of $G$.
\end{defn}

\begin{exa}
In the case where $\alpha = \aleph_0$ and where $\Ja = \mathfrak{Card}_{\aleph_0} = \aleph_0$ is the set of all finite cardinals, we recover the notion of centralizer defined in \cite[III.3]{Cohn}, which we call the \textbf{finitary centralizer}.  In \ref{para:set_ary_centralizer} we discuss the notion of \textbf{$\fSet$-ary centralizer} that arises by instead taking $\alpha = \kappa(\fSet)$ and $\Ja = \fSet$.
\end{exa}

\begin{para}\label{para:perp_gc}
The map $(-)^\perp$ from the powerset $\sP(\cO_A^{\,\alpha})$ to itself that is given by $G \mapsto G^\perp$ is order-reversing with respect to the inclusion order and satisfies $F \subseteq G^\perp\;\Leftrightarrow\;G \subseteq F^\perp$ for all $F,G \subseteq \cO^{\,\alpha}_A$.  It follows that the map $(-)^{\perp\perp}:\sP(\cO_A^{\,\alpha}) \rightarrow \sP(\cO_A^{\,\alpha})$ is order-preserving, idempotent, and inflationary, the latter term meaning that $F \subseteq F^{\perp\perp}$ for every $F \subseteq \cO_A^{\,\alpha}$.  It also follows that $F = F^{\perp\perp}$ iff $F = G^\perp$ for some $G \subseteq \cO_A^{\,\alpha}$.
\end{para}

\begin{para}\label{para:centralizer_clone}
Let $G$ be a set of $\Ja$-ary operations on a set $A$.  As noted in \ref{para:set_of_jary_ops}, $G$ is a $\Ja$-ary signature for which $A$ is canonically a $G$-algebra.  Hence, given any operation $f:A^J \rightarrow A$ with $J \in \Ja$, we may regard $A^J$ also as a $G$-algebra (namely the power $A^J$ in $G\textnormal{-ALG}$), and we readily deduce that  $f \in G^\perp$ iff $f:A^J \rightarrow A$ is a $G$-homomorphism.  This shows that for all $J \in \Ja$,
$$G^\perp(J) = G\textnormal{-ALG}(A^J,A)\;,$$
the set of all $G$-homomorphisms from $A^J$ to $A$.  It follows that $G^\perp$ is a $\Ja$-ary clone, which we therefore also call the \textbf{$\Ja$-ary centralizer clone} of $G$.  In particular, the $\Ja$-ary double centralizer $G^{\perp\perp}$ is a clone, and $G \subseteq G^{\perp\perp}$ by \ref{para:perp_gc}, so $\langle G\rangle_\alpha \subseteq G^{\perp\perp}$ and hence $G^\perp \subseteq \langle G\rangle_\alpha^\perp$ by \ref{para:perp_gc}, while the opposite inclusion $\langle G\rangle_\alpha^\perp \subseteq G^\perp$ holds simply because $G \subseteq \langle G \rangle_\alpha$.  This shows that
$$G^\perp = \langle G\rangle_\alpha^\perp\;.$$
\end{para}

\begin{para}\label{para:cc_of_alg}
Given a $\Ja$-ary algebra $(\Sigma,A)$, the the \textbf{$\Ja$-ary centralizer clone} of $(\Sigma,A)$ is the $\Ja$-ary centralizer $(\Sigma^A)^\perp$ of the set $\Sigma^A$ of basic $\Sigma$-operations of $A$.  Hence, $(\Sigma^A)^\perp(J) = \Sigma\textnormal{-ALG}(A^J,A)$ for each $J \in \Ja$ by \ref{para:centralizer_clone}.  We also call the $\Ja$-ary double centralizer $(\Sigma^A)^{\perp\perp}$ the \textbf{$\Ja$-ary double centralizer of $(\Sigma,A)$}.  By \ref{para:centralizer_clone}, $(\Sigma^A)^\perp$ is equally the centralizer $\langle \Sigma^A\rangle_\alpha^\perp$ of the $\Ja$-ary clone of operations $\langle \Sigma^A\rangle_\alpha$ of $(\Sigma,A)$, and hence $(\Sigma^A)^{\perp\perp} = \langle \Sigma^A\rangle_\alpha^{\perp\perp}$, so in particular $\langle \Sigma^A\rangle_\alpha \subseteq (\Sigma^A)^{\perp\perp}$.
\end{para}

\section{$\fSet$-ary extension and $\fSet$-ary centralizer clones}

In view of \ref{exa:small_arities}, the notion of \textit{$\fSet$-ary clone} is the special case of the notion of (concrete) clone in \ref{para:clone} that is obtained by considering arbitrary small sets as arities.  In the present section, we discuss certain interactions of the notions of $\fSet$-ary clone and $\Ja$-ary clone that enable the latter to be studied in terms of the former.

\begin{para}\label{para:restn}
Given a $\fSet$-ary clone $F \subseteq \cO_A$, we write $F|_\Ja$ to denote the intersection of $F$ with $\cO_A^{\,\alpha}$, so that $F|_\Ja$ is a set of $\Ja$-ary operations defined by $F|_\Ja(J) = F(J)$ $(J \in \Ja)$.  Clearly $F|_\Ja$ is a $\Ja$-ary clone, which we call $F|_\Ja$ the \textbf{($\Ja$-ary) restriction} of $F$.
\end{para}

\begin{para}
Given any (possibly large) set of $\fSet$-ary operations $F \subseteq \cO_A$ on a set $A$, we write $\langle F\rangle$ to denote the $\fSet$-ary clone generated by $F$ \pref{para:clone_gend_by}.  In particular, given a set of $\Ja$-ary operations $F \subseteq \cO_A^{\,\alpha}$ on $A$, we may regard $F \subseteq \cO_A$ in particular as a set of $\fSet$-ary operations on $A$ (with $F(X) = \emptyset$ for every small set $X \notin \Ja$), and so we may consider the $\fSet$-ary clone $\langle F \rangle$ generated by $F$.
\end{para}

\begin{para}\label{para:cl_pf}
If $f:A^Y \rightarrow A$ is an element of a given set of $\fSet$-ary operations $F \subseteq \cO_A$ and $\xi:Y \rightarrow X$ is any map, then the pushforward $\xi_*(f) = f \circ A^\xi:A^X\rightarrow  A$
is an element of the $\fSet$-ary clone $\langle F\rangle$ generated by $F$, since $A^\xi:A^X \rightarrow A^Y$ is the map induced by the projections $\pi_{\xi(y)}:A^X \rightarrow A$ $(y \in Y)$, which are elements of $\langle F \rangle$.  If $F$ happens to be a $\Ja$-ary clone, then $\langle F\rangle$ actually consists entirely of these pushforwards $\xi_*(f)$, as a consequence of the following (more general) result:
\end{para}

\begin{prop}
Let $F$ be a set of $\Ja$-ary operations on a set $A$.  Then the $\fSet$-ary clone $\langle F\rangle$ generated by $F$ consists of the operations obtained as pushforwards
\begin{equation}\label{eq:subn}\xi_*(g) = \bigl(A^X \xrightarrow{A^\xi} A^J \xrightarrow{g} A\bigr)\end{equation}
of operations $g \in \langle F\rangle_\alpha(J)$ along maps $\xi:J \rightarrow X$ with $J \in \Ja$ and $X \in \fSet$, where $\langle F\rangle_\alpha$ is the $\Ja$-ary clone generated by $F$.  That is,
$$\langle F\rangle(X) = \{\xi_*(g) \mid J \in \Ja,\,\xi \in X^J,\, g \in \langle F \rangle_\alpha(J)\}$$
for each small set $X$.
\end{prop}
\begin{proof}
Let $P$ denote the set of all $\fSet$-ary operations of the form \eqref{eq:subn}.  Then clearly $F \subseteq P$.  For each small set $X$, $P$ contains the projections $\pi_x:A^X \rightarrow A$ with $x \in X$ (since we may take $J = 1 \in \Ja$, $\xi = x$, and $g = \pi_0 = 1_A:A^1 = A \rightarrow A$).  To show that $P$ is, moreover, a $\fSet$-ary clone, let $p:A^Z \rightarrow A$ and $q_z:A^X \rightarrow A$ $(z \in Z)$ be elements of $P$, where $X$ and $Z$ are small sets; we must show that $p \circ (q_z)_{z \in Z}:A^X \rightarrow A$ lies in $P$.  There exist $K \in \Ja$, $\zeta \in Z^K$, and an operation $g:A^K \rightarrow A$ in $\langle F\rangle_\alpha$ with $p = g \circ A^\zeta$, and for each $z \in Z$ we may choose $J_z \in \Ja$, $\xi_z \in X^{J_z}$, and an operation $h_z:A^{J_z} \rightarrow A$ in $\langle F\rangle_\alpha$ with $q_z = h_z \circ A^{\xi_z}$.  Since $\alpha$ is a regular cardinal, there is a coproduct $J := \coprod_{k \in K}J_{\zeta(k)}$ in $\Set_{\Ja}$ that is preserved by the inclusion $\Set_{\Ja} \hookrightarrow \Set$ \pref{para:ar}, and we write $\iota_k:J_{\zeta(k)} \rightarrow J$ for the injections $(k \in K)$.  Hence we may consider the map $\xi:J \rightarrow X$ induced by the maps $\xi_{\zeta(k)}:J_{\zeta(k)} \rightarrow X$, so that $\xi \circ \iota_k = \xi_{\zeta(k)}$ $(k \in K)$.  For each $k \in K$ we obtain an operation $h'_{\zeta(k)} := h_{\zeta(k)} \circ A^{\iota_k}:A^J \rightarrow A$, which is an element of the clone $\langle F\rangle_\alpha$ since $h_{\zeta(k)}$ is so.   The composite $p \circ (q_z)_{z \in Z}:A^X \rightarrow A$ can be expressed also as the composite
$$A^X \overset{A^\xi}{\longrightarrow} A^J \xrightarrow{\left(h'_{\zeta(k)}\right)_{k \in K}} A^K \overset{g}{\longrightarrow} A\;,$$
which is an element of $P$ since $g \circ \left(h'_{\zeta(k)}\right)_{k \in K}$ is an element of the clone $\langle F\rangle_\alpha$.  This shows that $P$ is a $\fSet$-ary clone with $F \subseteq P$, so that $\langle F \rangle \subseteq P$.  Also, by \ref{para:restn}, $\langle F\rangle|_\Ja$ is a $\Ja$-ary clone and contains $F$, so that $\langle F\rangle_\alpha \subseteq \langle F \rangle|_\Ja$ and hence $\langle F\rangle_\alpha \subseteq \langle F \rangle$, which entails that each pushforward \eqref{eq:subn} is an element of the clone $\langle F \rangle$ , so that $P \subseteq \langle F \rangle$.
\end{proof}

\begin{cor}
Given any set of $\Ja$-ary operations $F$ on a set $A$, the $\Ja$-ary clone $\langle F\rangle_\alpha$ generated by $F$ is the $\Ja$-ary restriction $\langle F\rangle|_\Ja$  of the $\fSet$-ary clone $\langle F\rangle$ generated by $F$.
\end{cor}

\begin{cor}\label{thm:extn}
Let $F$ be a $\Ja$-ary clone on a set $A$. Then the $\fSet$-ary clone $\langle F \rangle$ consists of the composites $g \circ A^\xi$ with $J \in \Ja$, $\xi \in X^J$, and $g \in F(J)$.  In particular, $F$ is the $\Ja$-ary restriction $\langle F \rangle|_\Ja$ of $\langle F\rangle$.
\end{cor}

\begin{para}
Given a $\Ja$-ary clone $F$ on a set $A$, we call $\langle F \rangle$ the \textbf{($\fSet$-ary) extension} of $F$.  By \ref{thm:extn}, if we write $\Ja\textnormal{-Clone}(A)$ for set of all $\Ja$-ary clones on $A$, then the map $\Ja\textnormal{-Clone}(A) \rightarrow \fSet\textnormal{-Clone}(A)$ given by extension is a section of the map $\fSet\textnormal{-Clone}(A)\rightarrow \Ja\textnormal{-Clone}(A)$ given by restriction.
\end{para}

\begin{para}\label{para:set_ary_centralizer}
As a special case of \ref{defn:centralizer}, if $F$ is any set of $\fSet$-ary operations $F \subseteq \cO_A$ on a set $A$, then we can consider the $\fSet$-ary centralizer clone $F^{\perp(\fSet)}$ of $F$, as well as the $\fSet$-ary double centralizer clone $F^{\perp\perp(\fSet)}$, noting that $\langle F \rangle^{\perp(\fSet)} = F^{\perp(\fSet)}$ by \ref{para:centralizer_clone}.  Any set $F$ of $\Ja$-ary operations $F \subseteq \cO_A^{\,\alpha}$ is in particular a set of $\fSet$-ary operations, so we may again form the $\fSet$-ary centralizer clone $F^{\perp(\fSet)}$, whose $\Ja$-ary restriction is clearly the $\Ja$-ary centralizer $F^\perp$.
\end{para}

\section{A characterization of the centralizer clone}

In this section, we introduce one of the central techniques of this paper, which involves considering $\fSet$-ary operations of the form $A^{(A^X)} \rightarrow A$ for small sets $A$ and $X$, and we use this technique to prove a characterization of centralizers that we apply in later sections.  We begin with the following key observation:

\begin{prop}\label{thm:ev_matr}
Let $f:A^X \rightarrow A$ and $h:A^{(A^X)} \rightarrow A$ be operations, for small sets $A$ and $X$.  Let $f*h,f\stt h:A^{X \times A^X} \rightarrow A$ denote the first and second Kronecker products of $f$ and $h$ \pref{para:kr_pr}, and let $e \in A^{X \times A^X}$ be the \textbf{evaluation matrix}, $e = (a(x))_{(x,a) \in X \times A^X}$.  Then
$$(f*h)(e) = h(f)\;\;\;\;\text{and}\;\;\;\;(f\stt h)(e) = f(h \circ u_X)$$
where $u_X:X \rightarrow A^{(A^X)}$ is the map that sends each $x \in X$ to the projection $\pi_x$.
\end{prop}
\begin{proof}
With the notation of \ref{para:kr_pr}, $\acute{e} = u_X:X \rightarrow A^{(A^X)}$ and $\grave{e} = 1_{A^X}:A^X \rightarrow A^X$, so
$$(f*h)(e) = h(f \circ \grave{e}) = h(f)\;\;\text{and}\;\;(f\stt h)(e) = f(h \circ \acute{e}) = f(h \circ u_X)\;.$$
\end{proof}

The preceding observation is complemented by the following important result, which entails that we can encode Kronecker products with a fixed operation $g:A^Y \rightarrow A$ in terms of a family of pushforward operations of the form $A^{(A^X)} \rightarrow A$:

\begin{prop}\label{thm:pf_enc_first_kp}
Let $A$, $X$, $Y$ be small sets, let $g:A^Y \rightarrow A$ be an operation, let $a \in A^{X \times Y}$, and write
$$h_{ga} := \grave{a}_*(g) \;:\;A^{(A^X)} \longrightarrow A$$
to denote the pushforward of $g$ along the map $\grave{a}:Y \rightarrow A^X$.  Then
$$(f*g)(a) = h_{ga}(f) = (f * h_{ga})(e)$$
$$(f\stt g)(a) = f(h_{ga} \circ u_X) = (f \stt h_{ga})(e)$$
for every operation $f:A^X \rightarrow A$, where $e \in A^{X \times A^X}$ is the evaluation matrix and $u_X:X \rightarrow A^{(A^X)}$ is the map defined in \ref{thm:ev_matr}.
\end{prop}
\begin{proof}
Firstly, $(f * h_{ga})(e) = h_{ga}(f) = g(f \circ \grave{a}) = (f*g)(a)$ by \ref{thm:ev_matr} and \ref{para:kr_pr}.  Secondly, by \ref{thm:ev_matr}, $(f \stt h_{ga})(e) = f(h_{ga} \circ u_X)$, but $(h_{ga} \circ u_X)(x) = h_{ga}(\pi_x) = g(\pi_x \circ \grave{a}) = g(\acute{a}(x))$ for all $x \in X$ and hence $h_{ga} \circ u_X = g \circ \acute{a}$, so $(f \stt h_{ga})(e) = f(h_{ga} \circ u_X) = f(g \circ \acute{a}) = (f\stt g)(a)$ by \ref{para:kr_pr}.
\end{proof}

\begin{thm}\label{thm:charn_cc}
Let $G$ be a set of $\Ja$-ary operations on a small set $A$, and let \mbox{$f:A^J \rightarrow A$} be an operation with $J \in \Ja$.  The following are equivalent:
\begin{enumerate}
\item $f \in G^\perp$, i.e. $f$ is an element of the $\Ja$-ary centralizer clone of $G$;
\item $h(f) = f(h \circ u_J)$ for every operation $h:A^{(A^J)} \rightarrow A$ of arity $A^J$ in the $\fSet$-ary clone $\langle G\rangle$ generated by $G$, where $u_J:J \rightarrow A^{(A^J)}$ is defined as in \ref{thm:ev_matr};
\item $f$ commutes with every operation $h:A^{(A^J)} \rightarrow A$ of arity $A^J$ in $\langle G\rangle$.
\end{enumerate}
\end{thm}
\begin{proof}
Clearly $f \in G^\perp\;\Leftrightarrow\; f \in G^{\perp(\fSet)}$, while $G^{\perp(\fSet)} = \langle G\rangle^{\perp(\fSet)}$ by \ref{para:set_ary_centralizer}.  Hence (1) implies (3).  If (3) holds, then for each $h \in \langle G\rangle(A^J)$ we know that $f*h = f \stt h$, so in particular $h(f) = (f*h)(e) = (f\stt h)(e) = f(h \circ u_J)$ by \ref{thm:ev_matr}, where $e$ is the evaluation matrix, showing that (2) holds.  Lastly, suppose that (2) holds.  Let $g:A^K \rightarrow A$ be an operation in $G$.  Then, for all $a \in A^{J\times K}$, the operation $h_{ga}:A^{(A^J)} \rightarrow A$ of \ref{thm:pf_enc_first_kp} is by definition a pushforward of an operation $g \in G$, so $h_{ga} \in \langle G\rangle$ by \ref{para:cl_pf}, and hence $(f*g)(a) = h_{ga}(f) = f(h_{ga} \circ u_J) = (f\stt g)(a)$ by \ref{thm:pf_enc_first_kp} and (2), showing that $f*g = f\stt g$, i.e. that $f$ commutes with $g$.  Hence (1) holds.
\end{proof}

\begin{cor}\label{thm:charn_cc_cor}
Let $G$ be a set of $\Ja$-ary operations on a small set $A$.  Then $G^\perp$ is the set of all $\Ja$-ary operations $f:A^J \rightarrow A$ such that $h(f) = f(h \circ u_J)$ for every operation $h:A^{(A^J)} \rightarrow A$ of arity $A^J$ in the $\fSet$-ary clone $\langle G\rangle$ generated by $G$.
\end{cor}

A special case of \ref{thm:charn_cc_cor} for $\fSet$-ary operations is obtained by replacing $\Ja$ by $\fSet$.  As a consequence, we also obtain the following result for $\fSet$-ary \textit{clones} $G$:

\begin{cor}\label{thm:charn_cc_cor_set}
Let $G$ be a $\fSet$-ary clone on a small set $A$.  Then $G^{\perp(\fSet)}$ is the set of all $\fSet$-ary operations $f:A^X \rightarrow A$ such that $h(f) = f(h \circ u_X)$ for every operation $h:A^{(A^X)} \rightarrow A$ of arity $A^X$ in $G$.
\end{cor}

\begin{para}\label{para:cc_slc_preamble}
We can generalize the preceding result for $\fSet$-ary clones as follows, using the fact that if $\alpha$ is a strong limit cardinal, $A$ is an $\alpha$-small set, and $J \in \Ja$, then the power $A^J$ is $\alpha$-small and so may be identified with a set in $\Ja$, for which we shall use the same notation $A^J$.  With this identification, we obtain the following:
\end{para}

\begin{thm}\label{thm:cc_thm_for_slc}
Suppose that the regular cardinal $\alpha$ is also a strong limit cardinal, and let $G$ be a $\Ja$-ary clone on an $\alpha$-small set $A$.  Then $G^\perp$ is the set of all $\Ja$-ary operations $f:A^J \rightarrow A$ such that $h(f) = f(h \circ u_J)$ for every operation $h:A^{(A^J)} \rightarrow A$ of arity $A^J$ in $G$.
\end{thm}
\begin{proof}
Let $J \in \Ja$.  By \ref{thm:extn}, $G$ is the $\Ja$-ary restriction of the $\fSet$-ary clone $\langle G\rangle$, so that in the present case where $A^J \in \Ja$ we have that $G(A^J) = \langle G\rangle(A^J)$, and hence the result follows by \ref{thm:charn_cc_cor}.
\end{proof}

\begin{cor}\label{thm:charn_cc_cor_fin}
Let $G$ be a finitary clone on a finite set $A$ \pref{exa:fin_cl}.  Then the finitary centralizer $G^{\perp(\aleph_0)}$ is the set of all finitary operations $f:A^n \rightarrow A$ $(n \in \aleph_0)$ such that $h(f) = f\bigl(\bigl(h(\pi_i)\bigr)_{i \in n}\bigr)$ for every operation $h:A^{(A^n)} \rightarrow A$ of arity $A^n$ in $G$, where $\pi_i:A^n \rightarrow A$ denotes the projection $(i \in n)$.
\end{cor}

\section{Theorems on double centralizer clones in general}

We impose the following assumption throughout this section; this assumption is satisfied by the cardinals $\kappa(\fSet)$ and $\aleph_0$, for example:

\begin{assn}
We assume that the regular cardinal $\alpha$ is also a strong limit cardinal.
\end{assn}

\begin{thm}\label{thm:dbl_cc}
Let $(\Sigma,A)$ be a $\Ja$-ary algebra whose underlying set $A$ is $\alpha$-small.  Then the $\Ja$-ary double centralizer $(\Sigma^A)^{\perp\perp}$ is the set of all $\Ja$-ary operations $f:A^J \rightarrow A$ such that $h(f) = f(h \circ u_J)$ for every $\Sigma$-homomorphism $h:A^{(A^J)} \rightarrow A$.
\end{thm}
\begin{proof}
$(\Sigma^A)^{\perp\perp}$ is the $\Ja$-ary centralizer of the $\Ja$-ary clone $(\Sigma^A)^\perp$, so this follows from Theorem \ref{thm:cc_thm_for_slc} since $(\Sigma^A)^\perp(A^J) = \Sigma\textnormal{-ALG}(A^{(A^J)},A)$ by \ref{para:cc_of_alg}, with the identification in \ref{para:cc_slc_preamble}.
\end{proof}

Applying the preceding theorem to the cardinality $\kappa(\fSet)$ of the universe $\fSet$, we obtain the following result on the $\fSet$-ary double centralizer of an algebra:

\begin{cor}\label{thm:fset_ary_dbl_centralizer}
Let $(\Sigma,A)$ be a $\fSet$-ary algebra whose underlying set $A$ is small.  Then the $\fSet$-ary double centralizer $(\Sigma^A)^{\perp\perp(\fSet)}$ is the set of all $\fSet$-ary operations $f:A^X \rightarrow A$ such that $h(f) = f(h \circ u_X)$ for every $\Sigma$-homomorphism $h:A^{(A^X)} \rightarrow A$.
\end{cor}

If we instead apply Theorem \ref{thm:dbl_cc} to the regular strong limit cardinal $\aleph_0$, we obtain the following:

\begin{cor}\label{thm:finitary_dbl_centralizer}
Let $(\Sigma,A)$ be a finitary algebra whose underlying set $A$ is finite.  Then the finitary double centralizer $(\Sigma^A)^{\perp\perp(\aleph_0)}$ is the set of all finitary operations $f:A^n \rightarrow A$ such that $h(f) = f\bigl((h(\pi_i))_{i \in n}\bigr)$ for every $\Sigma$-homomorphism $h:A^{(A^n)} \rightarrow A$.
\end{cor}

The following corollary to Theorem \ref{thm:dbl_cc} provides a powerful method for addressing the following type of problem: \textit{Given a specific $\Ja$-ary algebra $(\Sigma,A)$, determine whether the $\Ja$-ary double centralizer clone $(\Sigma^A)^{\perp\perp}$ is equal to the clone of $\Ja$-ary derived operations $\langle \Sigma^A\rangle_\alpha$}.  In view of \ref{para:perp_gc}, this is equivalent to the following problem: \textit{Determine whether the clone of $\Ja$-ary derived operations $\langle \Sigma^A\rangle_\alpha$ is the $\Ja$-ary centralizer clone $G^\perp$ of some set $G$ of $\Ja$-ary operations on $A$.}

\begin{cor}\label{thm:method_for_showing_alg_sat}
Let $(\Sigma,A)$ be a $\Ja$-ary algebra whose underlying set $A$ is $\alpha$-small.  Then the following are equivalent: (1) The $\Ja$-ary double centralizer clone $(\Sigma^A)^{\perp\perp}$ is equal to the clone of \mbox{$\Ja$-ary} derived operations $\langle \Sigma^A\rangle_\alpha$; (2) for every $\Ja$-ary operation \mbox{$f:A^J \rightarrow A$}, if $f \notin \langle \Sigma^A\rangle_\alpha$ then there exists some $\Sigma$-homomorphism $h:A^{(A^J)} \rightarrow A$ such that $h(f) \neq f(h \circ u_J)$.
\end{cor}
\begin{proof}
This follows from \ref{thm:dbl_cc}, since we always have $\langle \Sigma^A\rangle_\alpha \subseteq (\Sigma^A)^{\perp\perp}$ by \ref{para:cc_of_alg}.
\end{proof}

Applying the preceding result to the cardinals $\kappa(\fSet)$ and $\aleph_0$, respectively, we obtain the following corollaries:

\begin{cor}\label{thm:dbl_cc_tech_for_setary}
Let $(\Sigma,A)$ be a $\fSet$-ary algebra whose underlying set $A$ is small.  Then the following are equivalent: (1) $(\Sigma^A)^{\perp\perp(\fSet)} = \langle \Sigma^A\rangle$; (2) for every $\fSet$-ary operation $f:A^X \rightarrow A$, if $f \notin \langle \Sigma^A\rangle$ then there exists some $\Sigma$-homomorphism $h:A^{(A^X)} \rightarrow A$ such that $h(f) \neq f(h \circ u_X)$.
\end{cor}

\begin{cor}\label{thm:dbl_cc_tech_for_finitary}
Let $(\Sigma,A)$ be a finitary algebra whose underlying set $A$ is finite.  Then the following are equivalent: (1) $(\Sigma^A)^{\perp\perp(\aleph_0)} = \langle \Sigma^A\rangle_{\aleph_0}$; (2) for every finitary operation $f:A^n \rightarrow A$, if $f \notin \langle \Sigma^A\rangle_{\aleph_0}$ then there exists some $\Sigma$-homomorphism $h:A^{(A^n)} \rightarrow A$ such that $h(f) \neq f\bigl((h(\pi_i))_{i \in n}\bigr)$.
\end{cor}

\begin{para}\label{para:abstract_perspective}
We shall see in Sections \ref{sec:vec}, \ref{sec:gset}, and \ref{sec:free_mon_actions} that the preceding results can be successfully applied to prove that $(\Sigma^A)^{\perp\perp} = \langle \Sigma^A\rangle_\alpha$ for several specific examples of $\Ja$-ary algebras with $\alpha = \kappa(\fSet)$ or $\alpha = \aleph_0$.  One important fact that helps make this possible is that for each arity $J \in \Ja$, the subset $\langle \Sigma^A\rangle_\alpha(J) \subseteq A^{(A^J)}$ can often be concretely characterized, as it is the $\Sigma$-subalgebra of $A^{(A^J)}$ generated by the projections $\pi_j:A^J \rightarrow A$ $(j \in J)$ (by \ref{thm:charn_clone_gend_by}).

Some insight into the preceding results can be obtained by applying the following abstract perspective; readers may skip this discussion and the resulting Corollary \ref{thm:pairwise_equalizer}, as these items are not used in the remainder of the paper.  Letting $J \in \Ja$, the map $u_J:J \rightarrow A^{(A^J)}$ $(j \in J)$ factors as a composite $u_J = (J \xrightarrow{\eta_J} \langle \Sigma^A\rangle_\alpha(J) \overset{\iota_J}{\hookrightarrow} A^{(A^J)})$ where $\iota_J$ is the inclusion and $\eta_J$ is given by $j \mapsto \pi_j$.  Since $\langle \Sigma^A\rangle_\alpha(J)$ is the $\Sigma$-subalgebra of $A^{(A^J)}$ generated by the projections $\pi_j$ $(j \in J)$, by \ref{thm:charn_clone_gend_by}, it follows that for every map $a:J \rightarrow A$, there is a unique $\Sigma$-homomorphism $a^\sharp:\langle \Sigma^A\rangle(J) \rightarrow A$ such $a^\sharp \circ \eta_J = a$, i.e. such that $a^\sharp(\pi_j) = a(j)$ for all $j \in J$; explicitly, $a^\sharp(f) = f(a)$ for every $f:A^J \rightarrow A$ in $\langle \Sigma^A\rangle_\alpha(J)$.  Hence, $a^\sharp$ is a restriction of the projection (or evaluation) map $\pi_a:A^{(A^J)} \rightarrow A$, which is also a $\Sigma$-homomorphism.  Thus, every $a \in A^J$ not only extends to a $\Sigma$-homomorphism $a^\sharp$ defined on $\langle \Sigma^A\rangle_\alpha(J)$ but also extends further to a $\Sigma$-homomorphism $\pi_a$ defined on $A^{(A^J)}$, so that we have a commutative diagram
$$
\xymatrix{
J \ar@/^3ex/[rr]^{u_J} \ar[r]_(.4){\eta_J} \ar@/_3ex/[drr]_a & \langle \Sigma^A\rangle_\alpha(J) \ar@/_2ex/[dr]^{a^\sharp} \ar[r]_{\iota_J} & A^{(A^J)} \ar[d]^{\pi_a},\\
& & A.
}
$$
But while the extension $a^\sharp$ is unique, the further extension $\pi_a$ need not be unique.  Thus we have a bijection $A^J \cong \Sigma\textnormal{-ALG}(\langle \Sigma^A\rangle_\alpha(J),A)$ given by $a \mapsto a^\sharp$, while on the other hand we have an injective map $s:\Sigma\textnormal{-ALG}(\langle \Sigma^A\rangle_\alpha(J),A) \rightarrow \Sigma\textnormal{-ALG}(A^{(A^J)},A)$ that is given by $s(a^\sharp) = \pi_a$ $(a \in A^J)$ and has a retraction $r:\Sigma\textnormal{-ALG}(A^{(A^J)},A) \rightarrow \Sigma\textnormal{-ALG}(\langle \Sigma^A\rangle_\alpha(J),A)$ given by restriction along $\iota_J$.  Hence the composite $s \circ r$ is an idempotent map from the set $\Sigma\textnormal{-ALG}(A^{(A^J)},A)$ to itself, and $s \circ r$ sends each $\Sigma$-homomorphism $h:A^{(A^J)} \rightarrow A$ to the $\Sigma$-homomorphism $(s \circ r)(h) = \pi_{h \circ u_J}:A^{(A^J)} \rightarrow A$ given by $f \mapsto f(h \circ u_J)$, which we call the \textit{regularization} of $h$, calling $h$ \textit{regular} if $h = \pi_{h \circ u_J}$.

With these notations we obtain the following equivalent formulation of Theorem \ref{thm:dbl_cc} and Corollary \ref{thm:method_for_showing_alg_sat}, which provides a categorical way of expressing that the $\Ja$-ary double centralizer $(\Sigma^A)^{\perp\perp}$ consists of all $\Ja$-ary operations $f:A^J \rightarrow A$ that `cannot tell the difference between a $\Sigma$-homomorphism $h:A^{(A^J)} \rightarrow A$ and its regularization'.  Here we employ the notion of the \textit{pairwise equalizer} \cite{Lu:Cmt} of a family of parallel pairs of morphisms $\phi_i,\psi_i:B \rightrightarrows C_i$ $(i \in I)$ in a category $\C$, which is by definition a limit (which may or may not exist) of the diagram in $\C$ consisting of these parallel pairs.  If $\C$ is the category of sets, $\SET$, or the category of $\Sigma$-algebras, $\Sigma\textnormal{-ALG}$, then the pairwise equalizer of such a family is the subset (or $\Sigma$-subalgebra) $E = \{b \in B \mid \forall i \in I\;:\;\phi_i(b) = \psi_i(b)\} \subseteq B$.
\end{para}

\begin{cor}\label{thm:pairwise_equalizer}
Let $(\Sigma,A)$ be a $\Ja$-ary algebra whose underlying set $A$ is $\alpha$-small.  Then for each $J \in \Ja$, the $\Sigma$-subalgebra $(\Sigma^A)^{\perp\perp}(J) \subseteq A^{(A^J)}$ is the pairwise equalizer
$$(\Sigma^A)^{\perp\perp}(J) = \bigl\{f \in A^{(A^J)} \mid \forall h \in \Sigma\textnormal{-ALG}(A^{(A^J)},A)\;:\;h(f) = \pi_{h \circ u_J}(f)\bigr\}$$
of the family consisting of the following parallel pairs of $\Sigma$-homomorphisms
\begin{equation}\label{eq:parallel_pairs}h,\;\pi_{h \circ u_J}\;:\;A^{(A^J)} \rightrightarrows A\;\;\;\;\;\;\bigl(h \in \Sigma\textnormal{-ALG}(A^{(A^J)},A)\bigr)\end{equation}
where $\pi_{h \circ u_J}$ is the regularization of $h$ \pref{para:abstract_perspective}.  Consequently, $\langle \Sigma^A \rangle_\alpha(J) = (\Sigma^A)^{\perp\perp}(J)$ iff the inclusion $\iota_J:\langle \Sigma^A \rangle_\alpha(J) \hookrightarrow A^{(A^J)}$ exhibits the $\Sigma$-algebra $\langle \Sigma^A \rangle_\alpha(J)$ as a pairwise equalizer of the family \eqref{eq:parallel_pairs}.
\end{cor}

\section{Double centralizer clones of $K$-vector spaces}\label{sec:vec}

We now begin our consideration of specific examples of applications of Corollary \ref{thm:method_for_showing_alg_sat} and its corollaries (\ref{thm:dbl_cc_tech_for_setary}, \ref{thm:dbl_cc_tech_for_finitary}), starting with the following basic example, in which we let $K$ be a field and let $V$ be a nonzero $K$-vector space.  We also assume that the sets $K$ and $V$ are small.  In this section, we take $\Sigma$ to be the usual signature for $K$-vector spaces, so that $\Sigma$ consists of a binary operation symbol $+$, a constant symbol $0$, a unary operation symbol $-$, and a unary operation symbol $c$ for each $c \in K$.  Since $V$ is a $K$-vector space, $V$ carries the structure of a $\Sigma$-algebra.  Hence $(\Sigma,V)$ is a finitary algebra and so may be regarded also as a $\fSet$-ary algebra.

Given a small set $X$, the $\Sigma$-algebra $\langle \Sigma^V\rangle(X)$ is the $K$-subspace of $V^{(V^X)}$ spanned by the projections $\pi_x:V^X \rightarrow V$ $(x \in X)$, which are linearly independent in $V^{(V^X)}$ since $V$ is nonzero.  Hence, the projections $\pi_x$ $(x \in X)$ constitute a basis for the $K$-vector space $\langle \Sigma^V\rangle(X)$.

\begin{thm}\label{thm:vec_sp_thm}
Given a small field $K$ and a small, nonzero $K$-vector space $V$, the \mbox{$\fSet$-ary} double centralizer clone $(\Sigma^V)^{\perp\perp(\fSet)}$ of $V$ is equal to the clone of $\fSet$-ary derived operations $\langle \Sigma^V\rangle$ of $V$.
\end{thm}
\begin{proof}
Let $f:V^X \rightarrow V$ be a (not necessarily linear) map, where $X$ is a small set, and suppose that $f \notin \langle \Sigma^V\rangle$.  The projections $\pi_x \in V^{(V^X)}$ $(x \in X)$ are linearly independent, and by assumption $f \in V^{(V^X)}\backslash \langle \Sigma^V\rangle$ lies outside the span of the $\pi_x$, so we may choose a basis $B \subseteq V^{(V^X)}$ for $V^{(V^X)}$ that is expressible as a disjoint union $B = \{\pi_x \mid x \in X\} \cup \{f\} \cup B_1$ for some set $B_1 \subseteq V^{(V^X)}$.  Also, note that upon applying $f$ to the constant zero map $0 \in V^X$ we obtain an element $f(0)$ of $V$, so since $V$ has at least two elements we may fix some element $v_1 \in V$ with $v_1 \neq f(0)$.  We may define a $K$-linear map $h:V^{(V^X)} \rightarrow V$ simply by defining $h$ on the basis $B$ as follows: $h(\pi_x) = 0$ for all $x \in X$, $h(f) = v_1$, and $h(b) = 0$ for every $b \in B_1$.   Recalling that $u_X:X \rightarrow V^{(V^X)}$ is given by $u_X(x) = \pi_x$, the composite $h \circ u_X:X \rightarrow V$ is therefore the zero element $h \circ u_X = 0 \in V^X$, so $h(f) = v_1 \neq f(0) = f(h \circ u_X)$.  The result now follows by Corollary \ref{thm:dbl_cc_tech_for_setary}.
\end{proof}

We can also prove the following finitary analogue when $K$ is a finite field:

\begin{thm}\label{thm:finite_field_vec}
Let $K$ be a finite field, and let $V$ be a nonzero, finite-dimensional $K$-vector space.  Then the finitary double centralizer clone $(\Sigma^V)^{\perp\perp(\aleph_0)}$ of $V$ is equal to the clone of finitary derived operations $\langle \Sigma^V \rangle_{\aleph_0}$ of $V$.
\end{thm}
\begin{proof}
We simply modify the proof of \ref{thm:vec_sp_thm} to consider finite cardinals $n$ rather than arbitrary small sets $X$, and we apply Corollary \ref{thm:dbl_cc_tech_for_finitary}.
\end{proof}

\section{Double centralizer clones of free $G$-sets}\label{sec:gset}

In our next example of an application of Corollary \ref{thm:method_for_showing_alg_sat}, we let $G$ be a group, written multiplicatively, and we let $A$ be a \textit{free $G$-set}, meaning that $A$ is a set equipped with an associative and unital left action $G \times A \rightarrow A$, written as $(\gamma,a) \mapsto\gamma a$, such that $A$ satisfies the following \textit{freeness axiom}: For all $\gamma \in G$ and $a \in A$, if $\gamma a = a$ then $\gamma = 1$.

The latter assumption on a $G$-set $A$ is equivalent to requiring that $A$ is free in the categorical sense, i.e. that there exists some subset $S \subseteq A$ such that $A$ has the universal property of a \textit{free $G$-set on $S$}, meaning that for every $G$-set $B$ and every map $b:S \rightarrow B$, there is a unique $G$-set homomorphism (i.e. $G$-equivariant map) $\hat{b}:A \rightarrow B$ such that $\hat{b} \circ \iota = b$, where $\iota:S \hookrightarrow A$ is the inclusion.  Indeed, if $A$ is any $G$-set and we write the distinct orbits in $A$ as $Ga_i = \{\gamma a_i \mid \gamma \in G\}$ $(i \in I)$ by choosing any family of elements $a_i \in A$ $(i \in I)$ representing these distinct orbits, then $A$ is a disjoint union of these orbits, and indeed a coproduct $A = \coprod_{i \in I} Ga_i$ in the category of $G$-sets (with their homomorphisms), noting that each $Ga_i$ is a sub-$G$-set of $A$.  Now supposing that $A$ satisfies the freeness axiom, there is clearly an isomorphism of $G$-sets $G \cong Ga_i$ given by $\gamma \mapsto \gamma a_i$ (where $G$ is regarded as a $G$-set under left multiplication).  Taking $S = \{a_i \mid i \in I\}$, it now follows that $A$ satisfies the universal property of a free $G$-set on $S$.

In the following, we write $\Sigma$ to denote the signature consisting of unary operation symbols $\gamma$ $(\gamma \in G)$.  With this notation, every $G$-set $A$ carries the structure of a $\Sigma$-algebra, so that $(\Sigma,A)$ is a finitary algebra and so may be regarded also as a $\fSet$-ary algebra.  Note that if $A$ is a free $G$-set and $Y$ is any nonempty set, then $A^Y$ is clearly a free $G$-set, with the pointwise $G$-action that makes $A^Y$ a power of $A$ in the category of $G$-sets.

\begin{thm}\label{thm:free_gset}
Let $A$ be a free $G$-set with at least two elements, where $G$ is a group, and suppose that $G$ and $A$ are small sets.  Then the $\fSet$-ary double centralizer clone $(\Sigma^A)^{\perp\perp(\fSet)}$ of $A$ is equal to the clone of $\fSet$-ary derived operations $\langle \Sigma^A\rangle$ of $A$.
\end{thm}
\begin{proof}
Let $X$ be a small set, let $f:A^X \rightarrow A$ be an arbitrary map, and suppose that $f \notin \langle \Sigma^A\rangle$.  Let us fix some element $a_0$ of the nonempty set $A$, and write $\overline{a_0} \in A^X$ to denote the constant map with value $a_0$.  Since $A$ has at least two elements, we may also fix some $a_1 \in A$ with $a_1 \neq f(\overline{a_0})$.

The set $A^X$ is nonempty, so $A^{(A^X)}$ is a free $G$-set, by the preceding remark.  The $\Sigma$-algebra $\langle \Sigma^A\rangle(X)$ is the sub-$G$-set of $A^{(A^X)}$ generated by the projections $\pi_x$ $(x \in X)$, i.e. the union
$$\langle \Sigma^A\rangle(X) = \{\gamma \pi_x \mid \gamma \in G, x \in X\} = \bigcup_{x \in X}G\pi_x$$
of the orbits $G\pi_x$ $(x \in X)$, which are distinct (and hence disjoint), by the following argument: Suppose that $G\pi_x = G\pi_y$ for distinct $x,y \in X$.  Then there is some $\gamma \in G$ with $\pi_x = \gamma\pi_y$.  Upon evaluating both sides at $\overline{a_0}$ we find that $a_0 = \gamma a_0$, and hence $\gamma = 1$ since $A$ is free.  Therefore $\pi_x = \pi_y$, and since $A$ has at least two elements this entails that $x = y$, a contradiction.

Since $f \in A^{(A^X)} \backslash \bigcup_{x \in X}G\pi_x$ and the $G\pi_x$ $(x \in X)$ are distinct, we may write the distinct orbits of the $G$-set $A^{(A^X)}$ as $Gs$ $(s \in S)$ for a subset $S \subseteq A^{(A^X)}$ chosen such that $f \in S$ and $\pi_x \in S$ for all $x \in X$.  Hence $A^{(A^X)}$ is a free $G$-set on $S$.  Therefore, by the universal property of this free $G$-set, we may define a $G$-set homomorphism $h:A^{(A^X)} \rightarrow A$ simply by defining $h$ on $S$, as follows: We define $h(s) = a_0$ for each $s \in S\backslash \{f\}$, while we define $h(f) = a_1$.  The resulting $G$-set homomorphism $h$ satisfies $h \circ u_X = \overline{a_0}:X \rightarrow A$, where $u_X:X \rightarrow A^{(A^X)}$ is the map given by $u_X(x) = \pi_x$ $(x \in X)$, so $h(f) = a_1 \neq f(\overline{a_0}) = f(h \circ u_X)$.  The result now follows, by Corollary \ref{thm:dbl_cc_tech_for_setary}.
\end{proof}

We can similarly prove the following analogous result for finite free $G$-sets:

\begin{thm}\label{thm:finite_gset}
Let $G$ be a finite group, and let $A$ be a finite free $G$-set with at least two elements.  Then the finitary double centralizer clone $(\Sigma^A)^{\perp\perp(\aleph_0)}$ of $A$ is equal to the clone of finitary derived operations $\langle \Sigma^A\rangle_{\aleph_0}$ of $A$.
\end{thm}
\begin{proof}
We can employ essentially the same proof as in \ref{thm:free_gset}, but with a finite cardinal $n$ in place of an arbitrary small set $X$, and instead invoke Corollary \ref{thm:dbl_cc_tech_for_finitary}.
\end{proof}

\section{Double centralizer clones for actions of free monoids}\label{sec:free_mon_actions}

Let $S$ be a small set, and let $S^*$ denote the free monoid on $S$, written multiplicatively.  Hence, the elements of $S^*$ are (possibly empty) strings  $s_1s_2...s_n$ consisting of elements $s_1,s_2,...,s_n \in S$, for various non-negative integers $n$.  The binary operation $\cdot$ on $S^*$ is concatenation, and the identity element $e$ is the empty string.

An \textbf{$S^*$-set} is, by definition, a set $A$ equipped with an associative, unital left action $S^* \times A \rightarrow A$, written as $(w,a) \mapsto wa$.  Equivalently, an $S^*$-set is a set $A$ equipped with a monoid homomorphism $S^* \rightarrow \End(A)$, where $\End(A)$ is the monoid of all maps from $A$ to itself, under composition.  Hence, by the universal property of the free monoid $S^*$, an $S^*$-set is equivalently given by a set $A$ equipped with a family of unary operations $s^A:A \rightarrow A$ $(s \in S)$.  $S^*$-sets are the objects of a category in which the morphisms are \textit{$S^*$-set homomorphisms}, i.e. maps that preserve the $S^*$-actions.

Given a set $Z$, an $S^*$-set $A$ is \textbf{free on $Z$} if $A$ is equipped with a map $\eta:Z \rightarrow A$ satisfying the following universal property: For every $S^*$-set $B$ and every map $b:Z \rightarrow B$ there is a unique $S^*$-set homomorphism $\hat{b}:A \rightarrow B$ such that $\hat{b} \circ \eta = b$.  We say that $A$ is \textbf{free} if $A$ is free on some set $Z$.  We may regard $S^*$ itself as an $S^*$-set under left-multiplication, and then $S^*$ is a free $S^*$-set on $1$, when equipped with the map $1 \rightarrow S^*$ that picks out the identity element $e$.  It follows that an $S^*$-set $A$ is free on $Z$ if $A$ is a coproduct $\coprod_{z \in Z} S^*$ in the category of $S^*$-sets, but coproducts in the category of $S^*$-sets are given by disjoint union, so that we may take the underlying set of this coproduct $\coprod_{z \in Z} S^*$ to be $S^* \times Z$.  Indeed, we can construct a free $S^*$-set on $Z$ by equipping the set $S^* \times Z$ with the $S^*$-action defined by $w(v,z) = (wv,z)$ for all $w \in S^*$ and $(v,z) \in S^* \times Z$, as the map $\eta:Z \rightarrow S^* \times Z$ defined by $z \mapsto (e,z)$ satisfies the above universal property.

In this section, we shall show that if $A$ is a small and free $S^*$-set with at least two elements, then the $\fSet$-ary centralizer clone of $A$ is equal to the clone of $\fSet$-ary operations on $A$.

Given any $S^*$-set $A$, we define a binary relation $\lt$ on the set $A$ by declaring for all $a,b \in A$ that $a \lt b$ if and only if there exists some $w \in S^*$ with $wa = b$.  This relation $\lt$ is clearly a \textit{preorder} on $A$, i.e. $\lt$ is reflexive and transitive.  Given an element $a \in A$, we write
$$\up a = \{b \in A \mid a \lt b\},\;\;\;\;\dn a = \{b \in A\ \mid b \lt a\}\;.$$
Note that $\up a$ is the sub-$S^*$-set $S^* a = \{wa \mid w \in S^*\}$ of $A$.

Let us say that an $S^*$-set $A$ \textbf{has unique transitions} provided that for all $w,v \in S^*$ and $a \in A$, if $wa = va$ then $w = v$.  Note that if $A$ has unique transitions, then for each $a \in A$, the map $S^* \rightarrow S^*a = \up a$ is an isomorphism of $S^*$-sets, whose inverse sends each $b \in \up a$ to the unique element $w_{ab} \in S^*$ such that
\begin{equation}\label{eq:wab}w_{ab}a = b\;.\end{equation}
Also, if $A$ has unique transitions, then the preorder $\lt$ on $A$ is a partial order, i.e. it is also antisymmetric, since if $a \leq b$ and $b \leq a$ in $A$ then $wa = b$ and $vb = a$ for some $w,v \in S^*$, so $vwa = a = ea$ and hence $vw = e$, but this entails that both $v$ and $w$ must be the empty string, i.e. $v = w = e$, so that $a = b$.

Given an $S^*$-set $A$ and a set $Y$, the set $A^Y$ carries the structure of an $S^*$-set, with the pointwise action, making $A^Y$ a power in the category of $S^*$-sets.  In the following result on $A^Y$ we use the following terminology: A \textbf{set-theoretic tree}\footnote{In set theory, the term \textit{tree} is used for such structures, though they can have multiple roots.} (cf. \cite[9.10]{Jech}) is a partially ordered set $(T,\lt)$ such that for each $t \in T$ the set $\dn t = \{u \in T \mid u \lt t\}$ is well-ordered by $\lt$, meaning that $\dn t$ is linearly ordered by $\lt$ and every nonempty subset of $\dn t$ has a least element.

\begin{prop}\label{thm:ay_set_theoretic_tree_w_unique_transitions}
Let $A$ be a free $S^*$-set, and let $Y$ be a nonempty set.  Then $A^Y$ is a set-theoretic tree, and $A^Y$ has unique transitions.
\end{prop}
\begin{proof}
$A$ is isomorphic to the free $S^*$-set $S^* \times Z$ constructed above, for some set $Z$, so since the statement to be proved is clearly invariant under isomorphism of $S^*$-sets (with respect to $A$), it suffices to treat the case where $A = S^* \times Z$.  Hence we may identify $A^Y$ with the set $(S^*)^Y \times Z^Y$, equipped with the $S^*$-action given by $w(\tau,\zeta) = (w\tau,\zeta)$ for all $w \in S^*$, $\tau \in (S^*)^Y$, and $\zeta \in Z^Y$, where we write the pointwise left action $S^* \times (S^*)^Y \rightarrow (S^*)^Y$ as $(w,\tau) \mapsto w\tau$, so that $(w\tau)(y) = w\tau(y)$ for all $y \in Y$.

Since $Y \neq \emptyset$, let us fix some $y_0 \in Y$.  To prove that $A^Y$ has unique transitions, let $\alpha = (\tau,\zeta) \in A^Y$ and suppose that $w\alpha = w'\alpha$ for some $w,w' \in S^*$.  Then $w\tau = w'\tau$ in $(S^*)^Y$, so $w \tau(y_0) = w'\tau(y_0)$ in $S^*$, and hence $w = w'$, simply because $S^*$ is a \textit{cancellative} monoid.  This shows that $A^Y$ has unique transitions, so by the preceding discussion, this entails that the preorder $\lt$ on $A^Y$ is a partial order.

Letting $\alpha = (\tau,\zeta) \in A^Y$, it suffices to show that the subset $\dn \alpha$ of $A^Y$ is well-ordered by $\lt$.  To show that $\dn \alpha$ is linearly ordered, let $\alpha_1 = (\tau_1,\zeta_1)$ and $\alpha_2 = (\tau_2,\zeta_2)$ be elements of $\dn \alpha$. Then $w_1\alpha_1 = \alpha = w_2\alpha_2$ for some $w_1,w_2 \in S^*$, and hence $w_1\tau_1 = \tau = w_2\tau_2$ and $\zeta_1 = \zeta = \zeta_2$.  In particular, writing $t_1 = \tau_1(y_0)$ and $t_2 = \tau_2(y_0)$, where $y_0$ is the element of $Y$ fixed above, we have $w_1t_1 = w_2t_2$.  Let us write $\mathsf{len}(w)$ for the length of each string $w \in S^*$.  First let us treat the case where $\mathsf{len}(w_1) \leq \mathsf{len}(w_2)$.  Then since both $w_1$ and $w_2$ are prefixes of $w_1t_1 = w_2t_2$ it follows that $w_1$ is a prefix of $w_2$, i.e. $w_2 = w_1v$ for some $v \in S^*$.  Hence, in $(S^*)^Y$, $w_1\tau_1 = w_2\tau_2 = w_1v \tau_2$, so for all $y \in Y$ we have that $w_1 \tau_1(y) = w_1v \tau_2(y)$ in $S^*$, and therefore  $\tau_1(y) = v\tau_2(y)$ by cancellation.  This shows that $\tau_1 = v\tau_2$, so $\alpha_1 = v\alpha_2$ and hence $\alpha_2 \lt \alpha_1$.  On the other hand, if instead $\mathsf{len}(w_2) \leq \mathsf{len}(w_1)$ then similarly $\alpha_1 \lt \alpha_2$.

This shows that $\dn \alpha$ is linearly ordered by $\lt$.  To show that $\dn \alpha$ is, moreover, well-ordered, let $\Gamma$ be a nonempty subset of $\dn \alpha$.  Suppose that the linearly ordered set $\Gamma$ does not have a least element.  Then, fixing any element $\gamma \in \Gamma$, we may choose a strictly decreasing chain $\gamma_1 = \gamma > \gamma_2 > \gamma_3 > ...$ in $\Gamma$.  For each natural number $n$, $\gamma_n = w_n\gamma_{n+1}$ for some nonempty string $w_n \in S^*$, so that $\mathsf{len}(w_n) \geq 1$.  Hence, for all $n$, $\gamma = w_1w_2...w_n\gamma_{n+1}$.  Writing $\gamma = (\tau,\zeta)$ and $\gamma_n = (\tau_n,\zeta_n)$ for all $n$ (so clearly $\zeta_n = \zeta$), we find that $\tau = w_1w_2...w_n\tau_{n+1}$.  Hence, for all $n$, the the length of the string $\tau(y_0) = w_1w_2...w_n\tau_{n+1}(y_0) \in S^*$ is at least $n$, a contradiction.
\end{proof}

In the following, we say that an element $r$ of an $S^*$-set $A$ is a \textbf{root} if $r$ is minimal (with respect to $\lt$), in the sense that for all $a \in A$, if $a \lt r$ then $a = r$.

\begin{prop}\label{thm:charn_free_sstar_set}
Let $A$ be an $S^*$-set.  The following are equivalent:
\begin{enumerate}
\item $A$ is a free $S^*$-set;
\item $A$ is a set-theoretic tree and has unique transitions;
\item $A$ has unique transitions, and for every element $a \in A$, there is a unique root $r_a$ in $A$ such that $r_a \lt a$.  
\end{enumerate}
If $A$ is free, then $A$ is a free $S^*$-set on the set $R(A)$ of all roots in $A$.
\end{prop}
\begin{proof}
Firstly, $(1)\Rightarrow(2)$ by \ref{thm:ay_set_theoretic_tree_w_unique_transitions}.  Next, if (2) holds, then for all $a \in A$ we know that the nonempty set $\downarrow a$ is well-ordered by $\lt$ and so has a unique least element $r_a$, which is therefore the unique root of $A$ with $r_a \lt a$, so (3) holds.  Lastly, suppose (3).  We claim that $A$ is isomorphic to the free $S^*$-set $S^* \times Z$ on the set $Z = R(A)$ of all roots in $A$.  By universal property of the latter free $S^*$-set, the inclusion map $\iota:R(A) \hookrightarrow A$ induces an $S^*$-set homomorphism $\hat{\iota}:S^* \times R(A) \rightarrow A$ given by $(w,r) \mapsto wr$.  It follows from (3) that this map $\hat{\iota}$ is a bijection, with inverse $A \rightarrow S^* \times R(A)$ given by $a \mapsto (w_{r_a,a},r_a)$ with the notation of \eqref{eq:wab}.
\end{proof}

\begin{cor}\label{thm:disj_union_roots}
Let $A$ be a free $S^*$-set.  Then $A$ is a disjoint union and a coproduct
$$A \;= \coprod_{r \in R(A)} \up r \;\cong \coprod_{r \in R(A)} S^*$$
of its disjoint sub-$S^*$-sets $\up r \cong S^*$ with $r \in R(A)$.
\end{cor}

\begin{cor}\label{thm:ay_free_sstar_set}
Let $A$ be a free $S^*$-set, and let $Y$ be a nonempty set.  Then $A^Y$ is a free $S^*$-set.
\end{cor}
\begin{proof}
This follows from \ref{thm:ay_set_theoretic_tree_w_unique_transitions} and \ref{thm:charn_free_sstar_set}.
\end{proof}

\begin{notn}\label{notn:w_a}
Let $a$ be an element of a free $S^*$-set $A$.  Then we write $w_a := w_{r_a,a}$ to denote the unique element of $S^*$ such that $w_a r_a = a$, where $r_a$ is the unique root in $A$ with $r_a \leq a$ \pref{thm:charn_free_sstar_set}.
\end{notn}

\begin{rem}\label{rem:left_cancellation}
Suppose that $wa = wb$ in a free $S^*$-set $A$, where $w \in S^*$ and $a,b \in A$.  Then $a = b$.  Indeed, writing $c = wa = wb$, we have $a,b \in \;\downarrow c$, so $r_a = r_c = r_b =: r$.  Hence $ww_ar = wa = wb = ww_b r$, so $ww_a = ww_b$ in $S^*$ since $A$ has unique transitions, and hence $w_a = w_b$.  Therefore $a = w_ar = w_br = b$ as needed.
\end{rem}

We shall require the following corollary to \ref{thm:charn_free_sstar_set} and \ref{thm:ay_free_sstar_set}:

\begin{cor}\label{thm:aax_free}
Let $A$ be a nonempty, free $S^*$-set, and let $X$ be a set.  Then $A^{(A^X)}$ is a free $S^*$-set, and every projection $\pi_x \in A^{(A^X)}$ $(x \in X)$ is a root in $A^{(A^X)}$.
\end{cor}
\begin{proof}
$A^{(A^X)}$ is a free $S^*$-set by \ref{thm:ay_free_sstar_set}.  Also, since $A \neq \emptyset$, \ref{thm:charn_free_sstar_set} entails that we may fix some root $r$ in $A$ and consider the constant family $\bar{r} = (r)_{x \in X} \in A^X$.  Letting $x \in X$, suppose that $\pi_x$ is not a root in $A^{(A^X)}$.  Then there exist some nonempty $w \in S^*$ and some $f \in A^{(A^X)}$ with $\pi_x = wf$, so that $r = \pi_x(\bar{r}) = wf(\bar{r})$ and hence $r > f(\bar{r})$ in $A$, contradicting the minimality of $r$.
\end{proof}

For the remainder of the section, let us write $\Sigma$ to denote the finitary signature consisting of unary operation symbols $s$ associated to the elements $s \in S$.  Then, by the discussion at the beginning of this section, $S^*$-sets may be described equivalently as $\Sigma$-algebras, and the category of $S^*$-sets may be identified with $\Sigma\textnormal{-ALG}$.

\begin{thm}\label{thm:free_sstar_set}
Let $A$ be a free $S^*$-set, and suppose that $A$ is small and has at least two elements.  Then the $\fSet$-ary double centralizer clone $(\Sigma^A)^{\perp\perp(\fSet)}$ of $A$ is equal to the clone of $\fSet$-ary derived operations $\langle \Sigma^A\rangle$ of $A$.
\end{thm}
\begin{proof}
In order to apply Corollary \ref{thm:dbl_cc_tech_for_setary}, let $f:A^X \rightarrow A$ be an arbitrary map, where $X$ is a small set, and suppose that $f \notin \langle \Sigma^A \rangle$.  By \ref{thm:charn_clone_gend_by}, $\langle \Sigma^A\rangle(X)$ is the sub-$S^*$-set of $A^{(A^X)}$ generated by the projections $\pi_x:A^X \rightarrow A$ $(x \in X)$, each of which is a root of $A^{(A^X)}$ by \ref{thm:aax_free}.  Hence, in view of \ref{thm:disj_union_roots}, $\langle \Sigma^A\rangle(X)$ is the union of the disjoint sub-$S^*$-sets $\uparrow \pi_x = S^*\pi_x$ $(x \in X)$ of $A^{(A^X)}$.  By \ref{thm:ay_free_sstar_set} and \ref{thm:charn_free_sstar_set}, there is a unique root $r_f$ in $A^{(A^X)}$ with $r_f \lt f$, and $r_f \notin \{\pi_x \mid x \in X\}$ since $f \notin \langle \Sigma^A\rangle$.  We may write $f = w_fr_f$ by \ref{notn:w_a}.

Let us fix an element $a_0 \in A$, and let $\bar{a}_0 \in A^X$ denote the constant family $\bar{a}_0 = (a_0)_{x \in X}$.  We then obtain an element $r_f(\bar{a}_0) \in A$, so since $A$ has at least two elements we may fix some $a_1 \in A$ with $a_1 \neq r_f(\bar{a}_0)$.

By \ref{thm:ay_free_sstar_set} and \ref{thm:charn_free_sstar_set}, $A^{(A^X)}$ is a free $S^*$-set on the set $R := R(A^{(A^X)})$ of all its roots, which we may write as a disjoint union $R = \{\pi_x \mid x \in X\} \cup \{r_f\} \cup R_1$ for some $R_1 \subseteq A^{(A^X)}$, by the above.  Hence, in order to define an $S^*$-set homomorphism $h:A^{(A^X)} \rightarrow A$, it suffices to define $h$ on the elements of $R$ as follows:  Define $h(\pi_x) = a_0$ for each $x \in X$, define $h(r_f) = a_1$, and define $h(r) = a_0$ for all $r \in R_1$.  Writing $u_X:X \rightarrow A^{(A^X)}$ for the map given by $x \mapsto \pi_x$, the resulting $S^*$-set homomorphism $h$ has $h \circ u_X = \bar{a}_0$, and we compute that $f(h \circ u_X) = f(\bar{a}_0) = w_fr_f(\bar{a}_0) \neq w_fa_1 = w_fh(r_f) = h(w_f r_f) = h(f)$ by \ref{rem:left_cancellation} since $r_f(\bar{a}_0) \neq a_1$. The result now follows by Corollary \ref{thm:dbl_cc_tech_for_setary}.
\end{proof}

Comparing and contrasting Theorem \ref{thm:free_sstar_set} on the $\fSet$-ary centralizer problem for free actions of a free monoid with Theorem \ref{thm:free_gset} on the $\fSet$-ary centralizer problem for free actions of a group, it is natural to ask whether these two theorems could be encompassed by a single result for free actions of a monoid satisfying suitable hypotheses; we leave this question open for future work.

\bibliographystyle{amsplain}
\bibliography{bib}
\end{document}